\documentclass[journal,twoside,web]{ieeecolor}

\usepackage{generic}
\usepackage{cite}
\usepackage{amsmath,amssymb,amsfonts}
\usepackage{algorithm}
\usepackage{algorithmic}
\usepackage{graphicx}
\usepackage{textcomp}
\usepackage{xcolor}
\allowdisplaybreaks

\newtheorem{thm}{Theorem}
\newtheorem{rem}[thm]{Remark}
\newtheorem{defn}[thm]{Definition}

\newtheorem{lem}[thm]{Lemma}

\def\BibTeX{{\rm B\kern-.05em{\sc i\kern-.025em b}\kern-.08em
    T\kern-.1667em\lower.7ex\hbox{E}\kern-.125emX}}
\begin{document}
\title{Temporal Parallelisation of Dynamic Programming and Linear Quadratic Control}
\author{Simo S\"arkk\"a,~\IEEEmembership{Senior Member,~IEEE,}
	\'Angel F. Garc\'ia-Fern\'andez
\thanks{S. S\"arkk\"a is with the Department
	of Electrical Engineering and Automation, Aalto University, 02150 Espoo, Finland (email: simo.sarkka@aalto.fi).}
\thanks{A. F. Garc\'ia-Fern\'andez is with the Department of Electrical Engineering and Electronics, University of Liverpool, Liverpool L69 3GJ, United Kingdom, and also with the ARIES Research Centre, Universidad Antonio de Nebrija, Madrid, Spain (email: angel.garcia-fernandez@liverpool.ac.uk). }}

\maketitle

\begin{abstract} This paper proposes a general formulation for temporal parallelisation of dynamic programming for optimal control problems. We derive the elements and associative operators to be able to use parallel scans to solve these problems with logarithmic time complexity rather than linear time complexity. We apply this methodology to problems with finite state and control spaces, linear quadratic tracking control problems, and to a class of nonlinear control problems. The computational benefits of the parallel methods are demonstrated via numerical simulations run on a graphics processing unit.
\end{abstract}


\section{Introduction}
\label{sec:introduction}
\IEEEPARstart{O}{ptimal} control theory (see, e.g.,  \cite{Stengel_book94, Lewis+Syrmos:1995, Bertsekas:2005}) is concerned with designing control signals to steer a system such that a given cost function is minimised, or equivalently, a performance measure is maximised. The system can be, for example, an airplane or autonomous vehicle which is steered to follow a given trajectory, an inventory system, a chemical reaction, or a mobile robot\cite{Stengel_book94,Biggs:2009,Komaee:2014,Zhao:2021,Sutton+Barto:2018}. 

Dynamic programming, in the form first introduced by Bellman 1950's, is a general method for determining feedback laws for optimal control and other sequential decision problems \cite{Bellman:1957,Bellman+Dreyfus:1962,Bertsekas:2005,Pakniyat:2017}, and it also forms the basis of reinforcement learning \cite{Sutton+Barto:2018}, which is a subfield of machine learning. The classic dynamic programming algorithm is a sequential procedure that proceeds backwards from the final time step to the initial time step, and determines the value (cost-to-go) function as well as the optimal control law in time complexity of $O(T)$, where $T$ is the number of time steps. The algorithm is optimal in the sense that no sequential algorithm that processes all the $T$ time steps can have a time-complexity less than $O(T)$.

However, the complexity $O(T)$ is only optimal in a computer with one single-core central processing unit (CPU). Nowadays, even general-purpose computers typically have multi-core CPUs with tens of cores and higher-end computers can have hundreds of them. Furthermore, graphics processing units (GPUs) have become common accessories of general-purpose computers and current high-end GPUs can have tens of thousands of computational cores that can be used to parallelise computations and lower the time-complexity. 

Dynamic programming algorithms that parallelise computations at each time step, but operate sequentially, are provided in \cite{Gengler:1996, Dormido-Canto05} for discrete states, and in \cite{Frison13} for the Riccati recursion in linear quadratic problems. Another approach to speed up computations for model predictive control (MPC) in linear quadratic problems is partial condensing \cite{Axehill15, Frison16}, which is based on splitting the problem into temporal blocks and eliminating the intermediate states algebraically. The required block-conversion can be done in parallel and the resulting modified linear quadratic problem can be solved using parallel methods such as in \cite{Frison13}. However, the complexity of the resulting algorithm is still linear in time.

The previous dynamic programming algorithms have linear time-complexity $O(T)$, but there are some approaches in literature to lower this complexity by using parallelisation across time. One idea applied in the context of an allocation process can be found in \cite[Sec. I.30]{Bellman+Dreyfus:1962}, where the time interval is divided into two, and the two problems are solved in parallel. Various forms of parallel algorithms for dynamic programming with discrete states are given in \cite{Casti73}. Reference \cite{Wright:1991} presents a partitioned dynamic programming suitable for parallelisation for linear quadratic control problems, though it has the disadvantage that some required inverse matrices may not exist. An iterated method for linear quadratic control problems, with constraints, in which each step can be parallelised is proposed in \cite{Frash15}, though it requires positive definite matrices in the cost function and may require regularisation. An algorithm to approximately solve an optimal control problem by solving different subproblems with partially overlapping time windows is provided in \cite{shin2019parallel}, and an approximate parallel algorithm for linear MPC is given in \cite{Jiang:2021}. Reference \cite{Calvet:1995} provides combination rules to separate the dynamic programming algorithm into different subproblems across the temporal domain. These combination rules are the foundation for temporal parallelisation.

The main contribution of this paper is to present a parallel formulation of dynamic programming that is exact and has a time complexity $O(\log T)$. None of the previous works achieve these two aspects simultaneously. The central idea is to reformulate dynamic programming in terms of associative operators, which enable the use of parallel scan algorithms \cite{Blelloch:1989,Blelloch:1990} to parallelise the algorithm. The resulting algorithm has a span-complexity of $O(\log T)$, which translates into a time-complexity of $O(\log T)$ with a large enough number of computational cores. The algorithm can therefore speed up the computations significantly for long time horizons. 

In this paper, we first provide the general formulation to parallelise dynamic programming by defining conditional value functions between two different time steps and combining them via the rule in  \cite{Calvet:1995}. We also show how to obtain the optimal control laws and resulting trajectories making use of parallel computation. Then, we explain how this general methodology can be directly applied to problems with finite state and control spaces. The second contribution of this paper is to specialise the methodology to linear quadratic optimal control problems, that is, to linear quadratic trackers (LQTs). The parallel LQT formulation is not straightforward, as it requires the propagation of the dual function \cite{Boyd_book04} associated with the conditional value function to avoid numerical problems. Our third contribution is to extend the parallel LQT algorithm to approximately solve certain nonlinear control problems by iterated linearisations, as in \cite{li2004iterative}. Finally, we have implemented these algorithms in TensorFlow \cite{Abadi_et_al:2015}, which enables parallel computations on GPUs, to experimentally show that the parallel algorithms provide a significant speed-up also in practice.

The present approach is closely related to the temporal parallelisation of Bayesian smoothers and hidden Markov model inference recently considered in \cite{Sarkka:2021,yaghoobi2021parallel,Hassan:2021}. These approaches use a similar scan-algorithm-based parallelisation in the context of state-estimation problems. The combination rule is also related to so-called max-plus algebras for dynamic programming which have been considered, for example, in \cite{Dower:2015,Zhang:2015,Zhang:2015b,Xu:2019}. 

The structure of the paper is the following. In Section~\ref{sec:background} we provide a brief background on dynamic programming and parallel computing, in Section~\ref{sec:Parallel_optimal_control} we provide the parallel methods to general and finite-state problems, in Section~\ref{sec:Parallel_LQT} we consider the parallel solution to LQT problems, in Section~\ref{sec:implementation}, we discuss some practical implementation aspects and computational complexity, in Section~\ref{sec:experimental} we experimentally illustrate the performance of the methods on a GPU platform, and finally we conclude the article in Section~\ref{sec:conclusion}. 

\section{Background} \label{sec:background}

We provide a brief background on optimal deterministic control and its dynamic programming solution in Section \ref{subsec:Deterministic_control}, the LQT case in Section \ref{subsec:LQT} and the parallel scan algorithm in Section \ref{subsec:Parallel_scan}.

\subsection{Deterministic control problem} \label{subsec:Deterministic_control}

We consider a deterministic control problem that consists of a difference equation and a cost function of the form \cite{LaValle_book06}
\begin{equation}
\begin{split}
    x_{k+1} &= f_k(x_k, u_k), \\
  C[u_{S:T-1}] &= \ell_T(x_T) + \sum_{n=S}^{T-1} \ell_n(x_n,u_n),
\end{split}
\label{eq:det_problem}
\end{equation}
where, for $k=S,\ldots,T$, $x_k$ is the state (typically $x_k \in \mathbb{R}^{n_x}$), $f_k(\cdot)$ is the function that models the state dynamics at time step $k$, $u_{S:T-1}=(u_S,\ldots,u_{T-1})$ is the control/decision sequence (typically $u_k \in \mathbb{R}^{n_u}$ with $n_u\leq n_x$), and $\ell_k(\cdot)$ is a lower bounded function that indicates the cost at time step $k$. The initial state $x_S$ is known. The aim is now to find a feedback control law or policy $u_k(x_k)$ such that if at step $k$ the state is $x_k$, the cost function $C[u_{S:T-1}]$ for the steps from $S$ to $T$ is minimized with the sequence $u_S(x_S), \ldots, u_T(x_T)$.

In Bellman's dynamic programming \cite{Bellman:1957,Bellman+Dreyfus:1962,Bertsekas:2005} the idea is to form a cost-to-go or value function $V_k(x_k)$ which gives the cost of the trajectory when we follow the optimal decisions for the remaining steps up to $T$ starting from state $x_k$. It can be shown \cite{Bellman:1957} that the value function admits the recursion
\begin{equation} \label{eq:Value_function}
\begin{split}
  V_k(x_k) = \min_{u_k} \left\{
    \ell_k(x_k,u_k) + V_{k+1}(f_k(x_k,u_k)) \right\}
\end{split}
\end{equation}
with $V_T(x_T) = \ell_T(x_T)$, which determines the optimal control law via
\begin{equation}
\begin{split}
  u_k(x_k) = \arg \min_{u_k} \left\{
    \ell_k(x_k,u_k) + V_{k+1}(f_k(x_k,u_k)) \right\}.
\end{split}
\label{eq:det_control}
\end{equation}

Given the control law \eqref{eq:det_control} for all time steps, we can compute the optimal trajectory from time steps $S+1$ to $T$, which is denoted as $(x^*_{S+1},\ldots,x^*_T)$, by an additional forward pass starting at $x^*_S=x_S$ and 
\begin{equation}\label{eq:optimal_trajectory}
x^*_{k+1} = f_k(x_k^*, u_k(x_k^*)) = f^*_k(x_k^*).
\end{equation}

\subsection{Linear quadratic tracker}\label{subsec:LQT}

The LQT problem \cite{Lewis+Syrmos:1995} is the solution to a linear quadratic control problem 
of the form 
\begin{equation}
\begin{split}
x_{k+1} &= F_kx_k+c_k+L_ku_k, \\
\ell_T(x_T) &=  \frac{1}{2}(H_Tx_T-r_T)^\top X_T(H_Tx_T-r_T), \\
\ell_n(x_n,u_n) &= \frac{1}{2}(H_nx_n-r_n)^\top  X_n(H_nx_n-r_n)
  + \frac{1}{2} u_n^\top U_nu_n, \label{eq:LQT_problem}
\end{split}
\end{equation}
for $n = S,\ldots, T-1$. We assume that  $X_n$ and $U_n$ are symmetric matrices such that $X_n\geq0$, $U_n>0$.

In this setting, the objective is that a linear combination of the states $H_kx_k$ follows a reference trajectory $r_k$ from time step $S$ to $T$. The linear quadratic regulator is a special case of the LQT problem by setting $r_k=0$, $c_k=0$, and $H_k = I \: \forall k$.

In this case, the value function is
\begin{equation}
V_k(x_k) = \mathrm{z}+ \frac{1}{2}x_k^\top S_k x_k - v_k^\top x_k, \label{eq:V_k_LQT}
\end{equation}
where  $v_k$ is an $n_x \times 1$ vector and $S_k$ is an $n_x \times n_x$ symmetric matrix and, throughout the paper, we use $\mathrm{z}$ to denote an undetermined constant that does not affect the calculations. The parameters $v_k$ and $S_k$ can be obtained recursively backwards. Starting with $v_T=H_T^\top X_Tr_N$ and $S_T=H_T^\top X_TH_T$, we obtain
\begin{align}
v_{k}&=\left(F_k-L_kK_{k}\right)^{\top}\left(v_{k+1}-S_{k+1}c_{k}\right)+H_k^{\top}X_kr_{k}, \label{eq:v_k_recursion}\\
S_{k}&=F_k^{\top}S_{k+1}(F_k-L_kK_{k})+H_k^{\top}X_kH_k, \label{eq:S_k_recursion}
\end{align}
where
\begin{equation}
K_{k}=	\left(L_k^{\top}S_{k+1}L_k+U_k\right)^{-1}L_k^{\top}S_{k+1}F_k
\end{equation}
and $k = S,\ldots, T-1$. 

The optimal control law is
\begin{equation}\label{Control_LQT}
u_{k}=	-K_{k}x_{k}+K_{k}^{v}v_{k+1}-K_{k}^{c}c_{k},
\end{equation}
where
\begin{align}
K_{k}^{v}&=\left(L_k^{\top}S_{k+1}L_k+U_k\right)^{-1}L_k^{\top},\\
K_{k}^{c}&=\left(L_k^{\top}S_{k+1}L_k+U_k\right)^{-1}L_k^{\top}S_{k+1}.
\end{align}

The optimal trajectory resulting from applying the optimal control law \eqref{Control_LQT} can be computed with an additional forward pass \eqref{eq:optimal_trajectory} starting at $x_S$ and with control law \eqref{Control_LQT}.

It should be noted that the derivation of LQT in \cite{Lewis+Syrmos:1995} does not include time-varying matrices or parameter $c_k$ in the problem formulation \eqref{eq:LQT_problem}, but it is straightforward to include these. It is also possible to use the LQT solution as a basis for approximate non-linear control by linearizing the system along a nominal trajectory (see \cite{Stengel_book94,Maybeck:1982b,li2004iterative} and Sec.~\ref{sec:nonlinear}). 

\subsection{Associative operators and parallel computing} \label{subsec:Parallel_scan}

Parallel computing (see, e.g., \cite{Rauber:2013,Barlas:2015}) refers to programming and algorithm design methods that take the availability of multiple computational cores into account. When some parts of the problem can be solved independently, then those parts can be solved in parallel to reduce the  computational time. The more parts we can solve in parallel, the more speed-up we get. 

Sequential problems which at first glance do not seem to be parallelisable can often be parallelised using so called parallel scan or all-prefix-sums algorithms \cite{Blelloch:1989,Blelloch:1990}. Given a sequence of elements $a_1,\ldots,a_T$ and an associative operator $\otimes$ defined on them, such as summation, multiplication, or minimisation, the parallel scan algorithm computes the all-prefix-sums operation which returns the values $s_1,\ldots, s_T$ such that
\begin{equation}
\begin{split}´
s_1 &= a_1, \\
s_2 &= a_1 \otimes a_2, \\
&\hdots \\
s_T &= a_1 \otimes a_3 \otimes \cdots \otimes a_T,
\end{split}
\end{equation}
in $O(\log T)$ time. The key aspect is that because the operator $\otimes$ is associative, we can rearrange the computations in various ways which generate independent sub-problems, for example,
\begin{equation}
  ((a_1 \otimes a_2) \otimes a_3) \otimes a_4
  = (a_1 \otimes a_2) \otimes (a_3 \otimes a_4),
\end{equation}
and, by a suitable combination of the partial solutions, we can obtain the result in $O(\log T)$ parallel steps. The specific combination requires an up-sweep and a down-sweep on a binary tree of computations \cite{Blelloch:1990}. A pseudocode is given in Algorithm \ref{alg:Parallel_scan}. Clearly, the prefix sums can also be computed in parallel in the backward direction $(a_1 \otimes \cdots \otimes a_T, \ldots ,a_{T-1} \otimes a_T, a_T)$. 

It should be noted that while parallel scans significantly lower the wall-clock time to compute all-prefix-sums, they have the drawback that the number of total computations is higher than in the sequential algorithm \cite{Blelloch:1990}. This implies that they require higher energy, which may not be suitable for small-scale mobile systems.

\begin{algorithm}
	\textbf{Input:} The elements $a_k$ for $k=1,\ldots,T$ and an associative operator $\otimes$.\\
	\textbf{Output:} The all prefix sums are returned in $a_k$ for $k=1,\ldots,T$.
	\begin{algorithmic}[1] 
		\STATE // Save the input:
		\FOR[Compute in parallel]{$i\gets1$ \textbf{to} $T$}
		\STATE $b_i \gets a_i$
		\ENDFOR
		\STATE // Up-sweep:
		\FOR{$d\gets0$ \textbf{to} $\log_2 T - 1$} 
		\FOR[Compute in parallel]{$i\gets0$ \textbf{to} $T-1$ \textbf{by} $2^{d+1}$}
		\STATE $j \gets i + 2^d$
		\STATE $k \gets i + 2^{d+1}$
		\STATE $a_k \gets a_j \otimes a_k$
		\ENDFOR
		\ENDFOR
		\STATE $a_T \leftarrow 0$ \COMMENT{Here, $0$ is the neutral element for $\otimes$}
		\STATE // Down-sweep:
		\FOR{$d\gets\log_2 T - 1$ \textbf{to} $0$}
		\FOR[Compute in parallel]{$i\gets0$ \textbf{to} $T-1$ \textbf{by} $2^{d+1}$}
		\STATE $j \gets i + 2^d$
		\STATE $k \gets i + 2^{d+1}$
		\STATE $t \gets a_j$
		\STATE $a_j \gets a_k$
		\STATE $a_k \gets a_k \otimes t$
		\ENDFOR
		\ENDFOR
		\STATE // Final pass:
		\FOR[Compute in parallel]{$i\gets1$ \textbf{to} $T$}
		\STATE $a_i \gets a_i \otimes b_i$
		\ENDFOR
	\end{algorithmic}
	\caption{Parallel-scan algorithm. The algorithm in this form assumes that $T$ is a power of $2$, but it can easily be generalized to an arbitrary $T$.}
	\label{alg:Parallel_scan}
\end{algorithm}

\section{Parallel optimal control}\label{sec:Parallel_optimal_control}

In this section, we start by defining conditional value functions and their combination rules (Section \ref{subsec:Cond_value_func}) and then we use them to define the associative operators and elements for parallelisation (Section \ref{subsec:Assoc_operator}). We also derive the parallel solution of the resulting optimal trajectory (Section \ref{subsec:Optimal_trajectory}), and finally, we discuss the case where the state space and controls take values in finite sets (Section \ref{subsec:finite_state}).

\subsection{Conditional value functions and combination rules} \label{subsec:Cond_value_func}

In this section, we present the conditional value functions and their combination rules, which are required to design the parallel algorithms.

\begin{defn}[Conditional value function]\label{Defn:conditional_value_function}
The conditional value function $V_{k \to i}(x_k, x_i)$ is the cost of the optimal trajectory starting from $x_k$ and ending at $x_i$, that is
\begin{equation}
\begin{split}
  V_{k\to i}(x_{k},x_{i})&=\min_{u_{k:i-1}}\sum_{n=k}^{i-1}\ell_{n}(x_{n},u_{n})
\end{split}
\label{eq:vki-definition}
\end{equation}
subject to
\begin{equation}\label{conditional_value_condition}
x_{n}=f_{n-1}(x_{n-1},u_{n-1})	\quad\forall n\in\{k+1,...,i\}.
\end{equation}
If there is no path connecting $x_k$ and $x_i$, then the constraint \eqref{conditional_value_condition} cannot be met and $V_{k\to i}(x_{k},x_{i})=\infty$.
\end{defn}

The combination rule for conditional value functions is provided in the following theorem.

\begin{thm} \label{the:v-comb}
The recursions for the value functions and conditional value functions can be written as
\begin{equation}
\begin{split}
  V_{k \to i}(x_k,x_i) = \min_{x_j} \left\{ V_{k \to j}(x_k,x_j) + V_{j \to i}(x_j,x_i) \right\},
\end{split}
\label{eq:vki-comb}
\end{equation}
for $k < j < i \le T$ and
\begin{equation}
\begin{split}
  V_k(x_k) = \min_{x_i} \left\{ V_{k \to i}(x_k,x_i) + V_i(x_i) \right\} ,
\end{split}
\label{eq:vk-comb}
\end{equation}
for $k < i \le T$.
\end{thm}

\begin{proof} 
See Appendix \ref{sec:Appendix}.
\end{proof} 

As part of the minimisation in \eqref{eq:vk-comb}, we also get the minimizing state $x_i$. Due to the principle of optimality, this value is the state at time step $i$ that is on the optimal trajectory from $x_k$ until time $T$. Similarly, the argument of minimisation $x_j$ in \eqref{eq:vki-comb} is part of the optimal trajectory from $x_k$ to $x_i$.

\subsection{Associative operator for value functions} \label{subsec:Assoc_operator}

The associative element $a$ of the parallel scan algorithm is defined to be a conditional value function $V_a(\cdot,\cdot):\mathbb{R}^{n_x} \times \mathbb{R}^{n_x}\rightarrow \mathbb{R}$ such that
\begin{equation}
\begin{split}
  a = V_a(x,y). \label{eq:element_conditional_value}
\end{split}
\end{equation}
The combination rule for two elements $a = V_a(x,y)$ and $b = V_b(x,y)$ is then given as follows.
\begin{defn} \label{def:gen-op}
Given elements $a$ and $b$ of the form \eqref{eq:element_conditional_value}, the binary associative operator for dynamic programming is
\begin{equation}\label{eq:operator_deterministic}
\begin{split}
  a \otimes b \triangleq \min_{z} \left\{ V_a(x,z) + V_b(z,y) \right\}.
\end{split}
\end{equation}
\end{defn}
This operator is an associative operator because $\min$ operation is associative. This is summarized in the following lemma.

\begin{lem}
The operator in Definition~\ref{def:gen-op} is associative.
\end{lem}
\begin{proof}
For three elements $a$, $b$, and $c$, we have
\begin{equation}
\begin{split}
  &(a \otimes b) \otimes c \\
  &\triangleq
  \min_{z'} \left\{ \min_{z} \left\{ V_a(x,z) + V_b(z,z') \right\} + V_c(z',y) \right\} \\
  &=
    \min_{z} \left\{ V_a(x,z) + \min_{z'} \left\{ V_b(z,z') + V_c(z',y)  \right\} \right\} \\
  &\triangleq
    a \otimes (b \otimes c),
\end{split}
\end{equation}
which shows that $(a \otimes b) \otimes c = a \otimes (b \otimes c)$.
\end{proof}
The elements and combination rule allows us to construct the conditional and conventional value functions as follows.
\begin{thm} \label{the:gen-init}
If we initialize the elements $a_k$ for $k=S,\ldots,T$ as
\begin{equation}
\begin{split}
  a_k &= V_{k \to k+1}(x_k,x_{k+1}),
\end{split}
\label{eq:gen-init}
\end{equation}
where $V_{T \to T+1}(x_T,x_{T+1}) \triangleq V_T(x_T)$, then
\begin{equation}
\begin{split}
  a_S \otimes a_{S+1} \otimes \cdots \otimes a_{k-1} &= V_{S \to k}(x_S,x_k) \label{eq:parallel_cond_V}
\end{split}
\end{equation}
and
\begin{equation}
\begin{split}
  a_k \otimes a_{i+1} \otimes \cdots \otimes a_T &= V_k(x_k). \label{eq:parallel_V_k}
\end{split}
\end{equation}
\end{thm}

\begin{proof}
Equation \eqref{eq:parallel_cond_V} results from the sequential application of \eqref{eq:vki-comb} forward, and \eqref{eq:parallel_V_k} from the sequential application \eqref{eq:vk-comb} backwards.
\end{proof}

Theorem~\ref{the:gen-init} implies that we can compute all value functions $V_k(\cdot)$ by initializing the elements as in \eqref{eq:gen-init}, using the associative operator in Definition \ref{def:gen-op} and computing  \eqref{eq:parallel_V_k} for $k=S,\ldots,T-1$, which corresponds to a (reverted) all-prefix-sum operation. Because the initialisation is fully parallelisable, we can directly use the parallel scan algorithm (see Algorithm \ref{alg:Parallel_scan}) to compute all value functions in  $O(\log T)$ parallel steps.

\begin{rem} \label{rem:gen-controls}
After computing all the value functions, we can obtain all the control laws $u_k(x_k)$ for $k=S,\cdots, T-1$ by using \eqref{eq:det_control}. This operation can be done in parallel for each $k$.
\end{rem}

\begin{rem} \label{rem:gen-xS-init2}
If we are interested in evaluating $V_{S \to k}(x_S,x_k)$ at a given $x_S$, as we are in trajectory recovery, then instead of first using \eqref{eq:parallel_cond_V} with initialisation 
\eqref{eq:gen-init}, and then evaluating the result at $x_S$, we can also initialise an extra element
\begin{equation}
\begin{split}
  a_{S-1} &= V_{S-1 \to S}(x,x') = \begin{cases}
    0 & \text{ if } x' = x_S, \\
    \infty & \text{ otherwise.}
  \end{cases}
\end{split}
\label{eq:gen-init2}
\end{equation}

\end{rem}

\subsection{Optimal trajectory recovery} \label{subsec:Optimal_trajectory}

Once we have obtained the optimal control laws \eqref{eq:det_control} in parallel, we can compute the resulting optimal trajectory $(x^*_{S+1},\ldots,x^*_T)$ in parallel using two methods. 

\subsubsection{Method 1}
In the first method for trajectory recovery, the state of the optimal trajectory at time step $k$ can be computed by using \eqref{eq:optimal_trajectory} and the composition of functions
\begin{equation}
x^*_{k}=\left(f_{k-1}^{*}\circ\ldots\circ f_{S+1}^{*}\circ f_{S}^{*}\right)(x_{S}). \label{eq:prediction_traj_composition}
\end{equation}
We can compute \eqref{eq:prediction_traj_composition} using parallel scans as follows. The associative element $a$ is defined to be a function on $x$, $a=f_a(\cdot)$ and the operator is the function composition in the following definition.

\begin{defn} \label{def:trajectory-op}
	Given elements $a=f_a(\cdot)$ and $b=f_b(\cdot)$, the binary associative operator for optimal trajectory recovery is 
	\begin{equation}
	\begin{split}
	a \otimes b \triangleq   f_b \circ f_a 
	\end{split}
	\end{equation}
	where $\circ$ denotes the composition of two functions, which is an associative operator \cite{Apostol_book67}. We should note that the order of the function composition is reverted.
\end{defn}
Then, we can recover the optimal trajectory via the following lemma.

\begin{lem} \label{lem:gen-trajectory}
	If we initialize element $a_S$ as the function $f^*_{S}(\cdot)$, which is given by \eqref{eq:optimal_trajectory}, evaluated at $x_S$
	\begin{equation}
	\begin{split}
	a_S &= f^*_{S}(x_S),
	\end{split}
	\end{equation}
	and, for $k=S+1,\ldots,T-1$, $a_k$ is initialised as the function
	\begin{equation}
	\begin{split}
	a_k &= f^*_{k}(\cdot),
	\end{split}
	\end{equation}
	then
	\begin{equation}
	\begin{split}
	a_{S} \otimes a_{S+1} \otimes \cdots \otimes a_{k-1}  &= x^*_k, \\
	\end{split}
	\end{equation}
	where $x^*_{k}$ is the state of the optimal trajectory at time step $k$.
\end{lem}
\subsubsection{Method 2}
An alternative method, which resembles the max-product algorithm in probabilistic graphical models \cite{Koller_book09}, is based on noticing that, from the definition of the conditional value function \eqref{eq:vki-definition} and the value function \eqref{eq:Value_function}, the state of the optimal trajectory at time step $k$ is given by
\begin{equation}\label{eq:trajectory_alternative}
x^*_k=\arg \min_{x_{k}} \{V_{S\to k}(x_{S},x_{k})+V_{k}(x_{k})\},
\end{equation}
where we recall that $x_{S}$ is the initial known state. That is, we can just minimise the sum of the forward conditional value function $V_{S\to k}(x_{S},x_{k})$ and the (backwards) value function $V_{k}(x_{k})$, which can be calculated using parallel scans via \eqref{eq:parallel_cond_V} and \eqref{eq:parallel_V_k}, respectively. Then, the minimisation \eqref{eq:trajectory_alternative} can be done for each node in parallel.

It should be noted that both approaches for optimal trajectory recovery require two parallel scans, one forward and one backwards, and one minimisation for each node.

\subsection{Finite state and control spaces} \label{subsec:finite_state}

The case in which the state and the control input belong to finite state spaces is important as we can solve the control problem in both sequential and parallel forms in closed-form. Let $x_{k}\in\left\{ 1,...,D_{x}\right\}$  and $u_{k}\in\left\{ 1,...,D_{u}\right\}$ where $D_x$ and $D_u$ are natural numbers. Then, $f_{k}\left(x_{k},u_{k}\right)$ and $\ell_{n}(x_{n},u_{n})$ can be represented by matrices of dimensions $D_{x}\times D_{u}$, $V_{k}(x_{k})$ by a vector of dimension $D_{x}$, $u_{k}(x_{k})$ by a vector of dimension $D_{x}$, and $V_{k\to i}(x_{k},x_{i})$ by a matrix of size $D_{x}\times D_{x}$. Due to the finite state space, the required minimisations in \eqref{eq:Value_function} and \eqref{eq:operator_deterministic}, can be performed by exhaustive search, which can also be parallelised.

For Method 1 for optimal trajectory recovery, the function $f_{k}^{*}(\cdot)$ can be represented as a vector of dimension $D_{x}$, and the function composition in \eqref{eq:trajectory_alternative} can be performed by evaluating all cases. 

\section{Parallel linear quadratic tracker} \label{sec:Parallel_LQT}

In this section, we provide the parallel solution to the LQT case. In Section \ref{subsec:Cond_value_LQT}, we derive the conditional value functions and combination rules. In Section \ref{subsec:Assoc_elements_LQT}, we derive the associative elements to obtain the value functions $V_{S \to k}(x_S,x_{k})$ and $V_{k}(x_{k})$. In Section \ref{sec:lqt_forward}, we address the computation of the optimal trajectory. Finally, in Section \ref{sec:extensions}, we discuss extensions to stochastic and non-linear problems.

\subsection{Conditional value functions and combination rules}\label{subsec:Cond_value_LQT}
For the LQT problem in \eqref{eq:LQT_problem}, $V_{k\to i}(x_{k},x_{i})$ in \eqref{eq:vki-definition} is a quadratic program with affine equality constraints \cite{Boyd_book04} that we represent by its dual problem
\begin{equation} \label{eq:Vki_dual}
V_{k\to i}(x_{k},x_{i})=\max_{\lambda}g_{k\to i}(\lambda;x_{k},x_{i}),
\end{equation}
where $\lambda$ is a Lagrange multiplier $n_x \times 1$ vector and the dual function $g_{k\to i}\left(\cdot,\cdot,\cdot\right)$ has the parameterisation  
\begin{align}
g_{k\to i}(\lambda;x_{k},x_{i})&=\mathrm{z}+\frac{1}{2}x_{k}^{\top}J_{k,i}x_{k}-x_{k}^{\top}\eta_{k,i}\nonumber\\
&-\frac{1}{2}\lambda^{\top}C_{k,i}\lambda-\lambda^{\top}\left(x_{i}-A_{k,i}x_{k}-b_{k,i}\right). \label{eq:g_ki}
\end{align}
If $C_{k,i}$ is invertible, one can solve \eqref{eq:Vki_dual}  by calculating the gradient of \eqref{eq:g_ki} with respect to $\lambda$ and setting it equal to zero, to obtain
\begin{align}
&V_{k\to i}(x_{k},x_{i})\nonumber\\
&=\mathrm{z}+\frac{1}{2}x_{k}^{\top}J_{k,i}x_{k}-x_{k}^{\top}\eta_{k,i}\nonumber\\
&+\frac{1}{2}\left(x_{i}-A_{k,i}x_{k}-b_{k,i}\right)^{\top}C_{k,i}^{-1}\left(x_{i}-A_{k,i}x_{k}-b_{k,i}\right). \label{eq:V_ki_invertible_C}
\end{align}

In this case, we can also interpret the  conditional value function \eqref{eq:V_ki_invertible_C} in terms of conditional Gaussian distributions as
\begin{align}
&\exp\left(-V_{k\to i}(x_{k},x_{i})\right)\nonumber\\
&\propto \mathrm{N}(x_i ; A_{k,i} x_k + b_{k,i}, C_{k,i}) \,
\mathrm{N}_I(x_k ; \eta_{k,i}, J_{k,i}), \label{eq:Gaussian_equi}
\end{align}
where $\mathrm{N}(\cdot ;\overline{x}, P)$ denotes a Gaussian density with mean $\overline{x}$ and covariance matrix $P$, and $\mathrm{N}_{I}\left(\cdot;\eta,J\right)$ denotes a Gaussian density parameterised in information form with information vector $\eta$ and information matrix $J$. A Gaussian distribution with mean $\overline{x}$ and covariance
matrix $P$ can be written in its information form as $\eta=P^{-1}\overline{x}$ and $J=P^{-1}$. 

Nevertheless, in general, $C_{k,i}$ is not invertible so it is suitable to keep the dual function parameterisation in \eqref{eq:Vki_dual}.

\begin{lem} \label{lem:LQT_combination}
Given two elements $V_{k \to j}(x_k,x_j)$ and $V_{j \to i}(x_j,x_i)$ of the form \eqref{eq:Vki_dual}, their combination $V_{k \to i}(x_k,x_i)$, which is obtained using Theorem \ref{the:v-comb}, is of the form  \eqref{eq:Vki_dual} and characterised by
\begin{equation}
\begin{split}
A_{k,i} &= A_{j,i} (I + C_{k,j} J_{j,i})^{-1} A_{k,j}, \\
b_{k,i} &= A_{j,i} (I + C_{k,j} J_{j,i})^{-1} (b_{k,j} + C_{k,j} \eta_{j,i}) + b_{j,i}, \\
C_{k,i} &= A_{j,i} (I + C_{k,j} J_{j,i})^{-1} C_{k,j} A_{j,i}^\top + C_{j,i}, \\
\eta_{k,i} &= A_{k,j}^\top (I + J_{j,i} C_{k,j})^{-1} (\eta_{j,i} - J_{j,i} b_{k,j}) + \eta_{k,j}, \\
J_{k,i} &= A_{k,j}^\top (I + J_{j,i} C_{k,j})^{-1} J_{j,i} A_{k,j} + J_{k,j}.
\end{split}
\label{eq:lqt_comb}
\end{equation}
where $I$ is an identity matrix of size $n_x$.
\end{lem}
The proof is provided in Appendix \ref{sec:LQT_combination_append}. It should be noted that the combination rule \eqref{eq:lqt_comb} is equivalent to the combination rule for the parallel linear and Gaussian filter, which also considers Gaussian densities of the form \eqref{eq:Gaussian_equi} \cite[Lem. 8]{Sarkka:2021}. 

\subsection{Associative elements to obtain the value functions}\label{subsec:Assoc_elements_LQT}

The following lemma establishes how to define the elements of the parallel scan algorithms to obtain the value functions $V_{S \to k}(x_S,x_{k})$ and $V_{k}(x_{k})$.

\begin{lem}\label{lem:LQT_assoc}
If we initialize the elements $a_k$ for $k=S,\ldots,T$ as:
\begin{equation}
\begin{split}
a_k &= V_{k \to k+1}(x_k,x_{k+1}),
\end{split}
\label{eq:LQT-init}
\end{equation}
where $V_{k \to k+1}(x_k,x_{k+1})$ is of the form  \eqref{eq:Vki_dual} with
\begin{equation}
\begin{split}
A_{k,k+1} &= F_k, \\
b_{k,k+1} &= c_k, \\
C_{k,k+1} &= L_k U_k^{-1} L_k^\top, \\
\eta_{k,k+1} &= H_k^\top X_k r_k, \\
J_{k,k+1} &= H_k^\top X_k H_k,
\end{split}
\label{eq:lqt_init}
\end{equation}
for $k=S,\cdots, T-1$ and $V_{T \to T+1}(x_T,x_{T+1})$ has parameters
\begin{equation}
\begin{split}
A_{T,T+1} &= 0, \\
b_{T,T+1} &= 0, \\ 
C_{T,T+1} &= 0, \\
\eta_{T,T+1} &= H_T^\top r_T, \\
J_{T,T+1} &= H_T^\top X_T H_T,
\end{split}
\end{equation}
then, 
\begin{equation}
\begin{split}
a_S \otimes a_{S+1} \otimes \cdots \otimes a_{k-1} &= V_{S \to k}(x_S,x_{k}) ,\label{eq:LQT_lem_V_S_k} \\
\end{split}
\end{equation}
and
\begin{equation}
\begin{split}
a_k \otimes a_{i+1} \otimes \cdots \otimes a_T &= V_{k\to T+1}(x_{k},x_{T+1}).\label{eq:LQT_lem_V_k}
\end{split}
\end{equation}
where
\begin{equation}
V_{k\to T+1}(x_{k},0)=V_{k}(x_{k}).\label{eq:LQT_lem_V_k2}
\end{equation}
Furthermore, $V_{k}(x_{k})$ is of the form \eqref{eq:V_k_LQT} with
\begin{equation}
\begin{split}
S_{k} &= J_{k,T+1}, \\
v_{k} &= \eta_{k,T+1}. \\
\end{split}
\end{equation}
\end{lem}

\begin{proof}
This lemma is proved in Appendix \ref{sec:Proof_parallel_LQT_append}.
\end{proof}

Once we obtain $v_{k+1}$ and $S_{k+1}$ using Lemma \ref{lem:LQT_assoc}, we can compute the optimal control $u_k$ using \eqref{Control_LQT}. 

\begin{rem} \label{lem:LQT_init2}
If we are interested in evaluating the conditional value functions $V_{S \to k}(x_S,x_{k})$ for a given $x_S$, then we can also directly initialise by
\begin{equation}
\begin{split}
A_{S-1,S} &= 0, \\
b_{S-1,S} &= x_S, \\
C_{S-1,S} &= 0, \\
\eta_{S-1,S} &= 0, \\
J_{S-1,S} &= 0.
\end{split}
\label{eq:lqt_init_remark}
\end{equation}
\end{rem}

\subsection{Optimal trajectory recovery} \label{sec:lqt_forward}

We proceed to explain how the two optimal trajectory recovery methods explained in Section \ref{subsec:Optimal_trajectory} work for the LQT problem.

\subsubsection{Method 1}

Plugging the optimal control law \eqref{Control_LQT} into the dynamic equation in \eqref{eq:LQT_problem}, the optimal trajectory function in \eqref{eq:optimal_trajectory} becomes
\begin{equation}
f_{k}^{*}(x_{k})=\widetilde{F}_{k}x_{k}+\widetilde{c}_{k},
\end{equation} 
where
\begin{align}
\widetilde{F}_{k}&=F_{k}-L_{k}K_{k}\\
\widetilde{c}_{k}&=c_{k}+L_{k}K_{k}^{v}v_{k+1}-L_{k}K_{k}^{c}c_{k}.
\end{align}
We denote a conditional optimal trajectory from time step $k$ to $i$ as
\begin{align}
f_{k\rightarrow j}^{*}(x_{k},x_{j})&=\left(f_{j-1}^{*}\circ\ldots\circ f_{k+1}^{*}\circ f_{k}^{*}\right)(x_{k})\\
&=\widetilde{F}_{k,j}x_{k}+\widetilde{c}_{k,j}. \label{eq:trajectory_element}
\end{align}
\begin{lem} \label{lem:LQT_trajectory_combination}
	Given two elements $f_{k\rightarrow j}^{*}(x_{k},x_{j})$ and $f_{j\rightarrow i}^{*}(x_{j},x_{i})$ of the form \eqref{eq:trajectory_element}, their combination $f_{k\rightarrow i}^{*}(x_{k},x_{i})$, given by Definition \ref{def:trajectory-op}, is a function $f_{k\rightarrow i}^{*}(x_{k},x_{i})$ of the form \eqref{eq:trajectory_element} with
	\begin{align}
	\widetilde{F}_{k,i}&=\widetilde{F}_{j,i}\widetilde{F}_{k,j}, \label{eq:F_trajectory}\\
	\widetilde{c}_{k,i}&=\widetilde{F}_{j,i}\widetilde{c}_{k,j}+\widetilde{c}_{j,i}. \label{eq:c_trajectory}
	\end{align}	
\end{lem}
\begin{proof}
The proof of this lemma is direct by using function compositions. 
\end{proof}

How to recover the optimal trajectory using parallel scans is indicated in the following lemma.
\begin{lem} \label{lem:lqt_fwd_init}
	If we initialise the elements of the parallel scan as $a_{k} = f_{k\rightarrow k+1}^{*}(x_{k})$, with
	\begin{align}
	\widetilde{F}_{k,k+1}&=\widetilde{F}_{k}, \\
	\widetilde{c}_{k,k+1}&=\widetilde{c}_{k},
	\end{align}
	for $k\in\left\{ S+1,...,T-1\right\}$ , and, for $k=S$, we set $\widetilde{F}_{S,S+1}=0$ and $\widetilde{c}_{S,S+1}=\widetilde{F}_{S}x_{S}+\widetilde{c}_{S}$, then, 
	\begin{equation}
	a_{S}\otimes a_{S+1}\otimes\cdots\otimes a_{k-1}=x^*_{k}
	\end{equation}
	where $x^*_{k}$ is the state of the optimal trajectory at time step $k$. 
\end{lem}

\subsubsection{Method 2}
This method makes use of \eqref{eq:trajectory_alternative} to recover the optimal trajectory. It first runs a forward pass to compute  $V_{S\to k}(x_{S},x_{k})$ and then a backward pass to compute  $V_{k}(x_{k})$. Then, the optimal trajectory is obtained via the following lemma.  
\begin{lem} \label{lem:LQT_trajectory_method2}
Given $V_{k}(x_{k})$ of the form \eqref{eq:V_k_LQT} and $V_{S\to k}(x_{S},x_{k})$  of the form \eqref{eq:Vki_dual}, the state of the optimal trajectory at time step $k$, which is obtained using \eqref{eq:trajectory_alternative}, is
\begin{align}
x^*_{k}&=\left(I+C_{S,k}S_{k}\right)^{-1}\left(A_{S,k}x_{S}+b_{S,k}+C_{S,k}v_{k}\right).
\label{eq:lqt_method_2}
\end{align}
\end{lem}
\begin{proof}
The proof is provided in Appendix~\ref{app:lqt_traj}.
\end{proof}

\subsection{Extensions} \label{sec:extensions}

In this section, the aim is to discuss some straightforward extensions of the parallel LQT.

\subsubsection{Extension to stochastic control}

Although the extension of the general framework introduced in this article to stochastic control problems is hard, the stochastic LQT case follows easily. Stochastic LQT is concerned with models of the form
\begin{align}
x_{k+1} &= F_kx_k+c_k+L_ku_k + G_k w_k, \\
  C[u_{S:T-1}] &= \mathrm{E}\left[ \ell_T(x_T) + \sum_{n=S}^{T-1} \ell_n(x_n,u_n) \right],
\label{eq:ext_cost_fun}
\end{align}
where $\ell_T(x_T)$ and $\ell_n(x_n)$ are as given in \eqref{eq:LQT_problem}, and $w_k$ is a zero mean white noise process with covariance $Q_k$, $G_k$ is a given matrix, and $\mathrm{E}\left[\cdot\right]$ denotes expectation over the state trajectory. It turns out that due to certainty equivalence property of linear stochastic control problems \cite{Stengel_book94,Maybeck:1982b}, the optimal control is still given by \eqref{Control_LQT} and the solution exactly matches the deterministic solution, that is, it is independent of $Q_k$ and $G_k$. The optimal value functions both in deterministic and stochastic cases have the form \eqref{eq:V_k_LQT}, but the value of the (irrelevant) constant is different.

It also results from the certainty equivalence property that the optimal control solution to the partially observed linear (affine) stochastic control problem with function \eqref{eq:ext_cost_fun} and dynamic and measurement models
\begin{align}
x_{k+1} &= F_kx_k+c_k+L_ku_k + G_k w_k, \\
y_k &= O_k x_k + d_k + e_k,
\end{align}
where $y_k$ is a measurement, $O_k$ is a measurement model matrix, $d_k$ is a deterministic bias, and $e_k$ is a zero mean Gaussian measurement noise, is given by \eqref{Control_LQT}, where the state $x_k$ is replaced with its Kalman filter estimate.

\subsubsection{Extension to more general cost functions}

Sometimes (such as in the nonlinear case below) we are interested in generalising the cost function in \eqref{eq:LQT_problem} to the following form for $n < T$:
\begin{equation} \label{eq:generalised_cost}
\begin{split}
\ell_n(x_n,u_n) &= \frac{1}{2}(H_n x_n - r_n)^\top  X_n (H_n x_n - r_n) \\
  &+ (H_n x_n - r_n)^\top M_n (u_n - s_n) \\
  &+ \frac{1}{2} (u_n - s_n)^\top U_n (u_n - s_n) \\
  &= \frac{1}{2} \begin{bmatrix}
    H_n x_n - r_n \\
    u_n - s_n
  \end{bmatrix}^\top
  \begin{bmatrix}
    X_n & M_n \\
    M_n^\top & U_n
  \end{bmatrix}
  \begin{bmatrix}
    H_n x_n - r_n \\
    u_n - s_n
  \end{bmatrix}.
  \end{split}
\end{equation}
We can now transform \eqref{eq:generalised_cost} into the form \eqref{eq:LQT_problem} using the factorisation
\begin{equation}
\begin{split}
 \begin{pmatrix}
  X_n & M_n \\
  M_n^\top & U_n
  \end{pmatrix} 
  &=
  \begin{pmatrix}
  I & 0 \\
  U_n^{-1} M_n^\top & I
  \end{pmatrix}^\top \\
  &\times
  \begin{pmatrix}
  X_n - M_n U_n^{-1} M_n^\top & 0 \\
  0 & U_n
  \end{pmatrix}
  \begin{pmatrix}
  I & 0 \\
  U_n^{-1} M_n^\top & I
  \end{pmatrix}.
\end{split}
\end{equation}
We thus have
\begin{equation}
\begin{split}
  &\begin{pmatrix}
  I & 0 \\
  U_n^{-1} M_n^\top & I
  \end{pmatrix}
  \begin{pmatrix}
       H_n x_n - r_n \\
       u_n - s_n
   \end{pmatrix} \\
   &= 
  \begin{pmatrix}
       H_n x_n - r_n \\
       U_n^{-1} M_n^\top (H_n x_n - r_n) + u_n - s_n
  \end{pmatrix},
\end{split}
\end{equation}
and by defining
\begin{equation}
\begin{split}
  \tilde{u}_n &= U_n^{-1} M_n^\top (H_n x_n - r_n) + u_n - s_n, \\
  \tilde{F}_n &= F_n - L_n U_n^{-1} M_n^\top H_n, \\
  \tilde{c}_n &= c_n + L_n U_n^{-1} M_n^\top r_n + L_n s_n, \\
  \tilde{X}_n &= X_n - M_n U_n^{-1} M_n^\top, \\
  \tilde{U}_n &= U_n,
\end{split}
\end{equation}
we get a system of the form \eqref{eq:LQT_problem}. This system can then be solved for $(x_n,\tilde{u}_n)$, and the final control signal can be recovered via
\begin{equation}
  u_n = \tilde{u}_n - U_n^{-1} M_n^\top (H_n x_n - r_n) + s_n. \\
\end{equation}

\subsubsection{Extension to nonlinear control} \label{sec:nonlinear}
The equations for solving the LQT problem can be extended to approximately solve nonlinear LQT systems by performing iterated linearisations, as in \cite{li2004iterative}. Let us consider a system of the form
\begin{equation}\label{eq:nonlinear_system}
\begin{split}
x_{k+1} & =f_{k}(x_{k},u_{k}),\\
\ell_{n}(x_{n},u_{n}) & =\frac{1}{2}(h_{n}(x_{n})-r_{n})^{\top}X_{n}(h_{n}(x_{n})-r_{n})\\
& \quad+\frac{1}{2}(g_{n}(u_{n})-s_{n})^{\top}U_{n}(g_{n}(u_{n})-s_{n}),\\
\ell_{T}(x_{T}) & =\frac{1}{2}(h_{T}(x_{T})-r_{T})^{\top}X_{T}(h_{T}(x_{T})-r_{T}),
\end{split}
\end{equation}
where $f_{k}(\cdot)$, $g_{n}(\cdot)$ and $h_{n}(\cdot)$ are possibly nonlinear functions. 
Given a nominal trajectory $\ensuremath{\bar{x}_{k},\bar{u}_{k}}$ for $k\in{S,...,T}$, we can linearise the nonlinear functions using first-order Taylor series as
\begin{equation}\label{eq:nonlinear_linarisation}
\begin{split}
f_{k}(x_{k},u_{k})&\approx f_{k}(\bar{x}_{k},\bar{u}_{k})+J_{f_{k}}^{x}\,(x_{k}-\bar{x}_{k})+J_{f_{k}}^{u}\,(u_{k}-\bar{u}_{k})\\
h_{n}(x_{n})&\approx h_{n}(\bar{x}_{n})+J_{h_{n}}^{x}\,(x_{n}-\bar{x}_{n}),\\
g_{n}(u_{n})&\approx g_{n}(\bar{u}_{n})+J_{g_{n}}^{u}\,(u_{n}-\bar{u}_{n}),
\end{split}
\end{equation}
where $J_{f_{k}}^{x}$ represents the Jacobian of function $f_{k}(\cdot)$ evaluated at $\bar{x}_{k},\bar{u}_{k}$ with respect to variable $x$.

Starting with a nominal trajectory $\ensuremath{\bar{x}_{k}^{1},\bar{u}_{k}^{1}}$ for $k\in{S,\ldots,T}$, we linearise the system using \eqref{eq:nonlinear_linarisation}, obtain the value functions using parallel scans (see Lemma \ref{lem:LQT_assoc}), and obtain a new optimal trajectory $\bar{x}_{k}^{2}$ and control $\bar{u}_{k}^{2}$. Then, we can repeat this procedure of linearisation and optimal trajectory/control computation until convergence. The procedure may be initialised, for example, with $\bar{x}_{k}^{1}=0,\bar{u}_{k}^{1}=0$ or $\bar{x}_{k}^{1}=x_{S},\bar{u}_{k}^{1}=0$ $\forall k$.

\section{Implementation and computational complexity} \label{sec:implementation}

In this section, we first discuss the practical implementation of the methods in Section \ref{subsec:practical_implementation}. We then analyse the computational complexity in Section \ref{subsec:Computational_complexity}. Finally, we explain how to perform parallelisation in blocks in Section \ref{sec:block-processing}.

\subsection{Practical implementation of parallel control}\label{subsec:practical_implementation}

Given the associative operators and the elements, the solutions to the dynamic programming and trajectory prediction problems reduce to an initialisation step followed by a single call to a parallel scan algorithm routine parameterised by these operators and elements. Given the result of the scan, there can also be a result-extraction step which computes the final optimal control from the scan results. For example, the LQT control law computation consists of the following steps:
\begin{enumerate}
\item \emph{Initialisation:} Compute the elements $A_{k,k+1}$, $b_{k,k+1}$, $C_{k,k+1}$, $\eta_{k,k+1}$, and $J_{k,k+1}$ defined in Lemma~\ref{lem:LQT_assoc} for all $k$ in parallel.

\item \emph{Parallel scan:} Call the backward parallel scan routine and, as its arguments, give the initialised elements above along with pointer to the operator in Lemma~\ref{lem:LQT_combination}. This returns $V_{k}(x_{k})$ for all $k$, see \eqref{eq:LQT_lem_V_k} and \eqref{eq:LQT_lem_V_k2}.

\item \emph{Extraction:} Compute the optimal control using \eqref{Control_LQT} in parallel for all $k$.
\end{enumerate}

The control law for a finite-state control problem is initialised with the conditional value functions in Theorem~\ref{the:gen-init} and the parallel scan routine is given a pointer to the operator in Definition~\ref{def:gen-op}. The control law computation is finally done with \eqref{eq:det_control} using the value functions computed in parallel. 

Sometimes, we also need to compute the actual trajectory and the corresponding optimal controls forward in time. For example, in iterative non-linear extensions of LQT discussed in Section~\ref{sec:nonlinear} we need to linearise the trajectory with respect to the optimal trajectory and control obtained at the previous iteration. In this case, after computing the control laws, we need to do another computational pass. For example, in the LQT case when using Method 1 from Section~\ref{sec:lqt_forward}, we do the following:
\begin{enumerate}
\item \emph{Initialisation:} Compute the elements $\widetilde{F}_{k,k+1}$ and $\widetilde{c}_{k,k+1}$  using Lemma~\ref{lem:lqt_fwd_init} for all $k$ in parallel.

\item \emph{Parallel scan:} Call the forward parallel scan routine and, as its arguments, give the initialised elements above along with pointer to the operator given in Lemma~\ref{lem:LQT_trajectory_combination}.

\item \emph{Extraction:} The optimal trajectory can be extracted from the forward scan results as the elements $\widetilde{c}_{S,k}$, see Lemma \ref{lem:lqt_fwd_init}.
\end{enumerate}
The steps for Method 1 in the finite-state case are analogous, but the elements are initialised according to Lemma~\ref{lem:gen-trajectory} and the combination operator is given in Definition~\ref{def:trajectory-op}. 

When using Method 2 for optimal trajectory recovery, the operator is the same as in the backward computation for the control law. The parallel scan is done in the forward direction and the final results still need to be evaluated at $x_S$ unless initialisation is done using Remark~\ref{lem:LQT_init2}. Furthermore, after computing the backward and forward scans, we still need to compute the optimal states by using \eqref{eq:trajectory_alternative}, which in the case of LQT reduces to \eqref{eq:lqt_method_2}

\subsection{Computational complexity}\label{subsec:Computational_complexity}

We proceed to analyse the computational complexity of the proposed methods. For this purpose it is useful to assume that the computer that we have operates according to the PRAM (parallel random access machine) model of computation (see, e.g.,  \cite{Rauber:2013,Barlas:2015}). In this model, we assume that we have a bounded number $P$ of identical processors controller by a common clock with a read/write access to a shared random access memory. This model is quite accurate for multi-core CPUs and GPUs.

For simplicity of analysis, we assume that the number of processors is large enough (say $P \to \infty$), so that the number of processors does not limit the parallelisation. Thus, the parallel scan algorithm has a time-complexity (i.e, span-complexity) of $O(\log T)$ in the number of associative operations. In the following, we also take the dimensionality of the state into account and therefore the time-complexities not only depend on the number of time steps $T$, but also on the number of states $D_x$ and number of controls $D_u$ in finite state-space case, and dimensionalities of the state $n_x$ and control $n_u$ in the LQT case. We analyse complexity using $O(\cdot)$ notation, as it enables us to analyse computational complexity avoiding low-level operation details that are not relevant to this contribution \cite{Arora_book07}.

The combination rule computations can also be parallelised, and their complexity will therefore depend on whether they are parallelised or not. For example, if we do not parallelise the computations in the LQT combination rule given in Lemma~\ref{lem:LQT_combination}, then their complexity is $O(n_x^3)$ due to the matrix inverses (or equivalent LU-factorisations) involved. If we perform the parallelisation, these LU-factorisations can be performed in parallel in $O(n_x)$ span time \cite{Ortega:1988}. The following analysis is based on assuming that we in fact use parallel matrix routines to implement the operations at the combination steps, along with all the other steps.

For the parallel algorithms we obtain the following results.

\begin{lem}\label{lem:PRAM_finite}
In a PRAM computer with large enough number of processors ($P \to \infty$) and the finite-state control problem, the span time complexity of
\begin{itemize}
\item computing value functions and the control law is $O(\log D_u + (\log T) \, (\log D_x))$; 
\item recovering the trajectory is $O(\log T)$ with Method 1, and $O(\log D_u + \log D_x + (\log T) \, (\log D_x))$ with Method 2. 
\end{itemize}
\end{lem}

\begin{proof}
The initialisation of the value function computation is done using \eqref{eq:gen-init} which has a time (span) complexity of $O(\log D_u)$ due to the minimisation operation over the control input. The associative operator in Definition \ref{def:gen-op} is fully parallelisable in summation, but the span complexity of the minimisation over the state is $O(\log D_x)$. The control law computation \eqref{eq:det_control} also has the complexity $O(\log D_u)$ and hence the total span complexity follows. For trajectory recovery with Method 1, we notice that each of the steps of initialisation and associative operator application are fully parallelisable. In Method 2, the initialisation and value function computation have the same complexity as in the backward value function computation and the minimisation at the final combination step takes $O(\log D_x)$ time.
\end{proof}

\begin{lem}\label{lem:PRAM_LQT}
In a PRAM computer with large enough number of processors ($P \to \infty$) and the LQT problem, the span time complexity of
\begin{itemize}
\item computing value functions and the control law is $O(n_u + n_x \, \log T)$;
\item recovering the trajectory is $O(\log n_u + (\log T) \, (\log n_x))$ with Method 1, and $O(n_x + n_x \, \log T)$ with Method 2. 
\end{itemize}
\end{lem}

\begin{proof}
A product of $n \times n$ matrices can be computed in parallel in $O(\log n)$ time, and an $n \times n$ LU factorisation can be computed in parallel in $O(n)$ time \cite{Ortega:1988}. Hence the initialisation requires $O(n_u)$ time as the contribution of the matrix products is negligible. The time complexity of the associative operator is dominated by the LU factorisations which take $O(n_x)$ time and the matrix factorisations at the control law computation can be performed in $O(n_u)$ time. The matrix products required in initialisation have negligible effect and therefore, the total complexity follows. In trajectory recovery Method 1, the matrix products at the initialisation can be computed in $O(\log n_u)$ time and the associative operators in $O(\log n_x)$ time. In Method 2, the initialisation is again negligible, and associative operations take $O(n_x)$ time, and the final combination $O(n_x)$ time.
\end{proof}

It should be noted that, according to Lemmas \ref{lem:PRAM_finite} and \ref{lem:PRAM_LQT}, Method 1 is computationally more efficient than Method 2 for large $D_u$ or large $D_x$ in the discrete case, and for large $n_x$ and $n_u$ in the LQT case. Nevertheless, a benefit of Method 2 is that it can be run in parallel with the backward pass.

Although the above analysis results give a guideline for performance in large number of processors (computational cores), with finite number of processors, we can expect worse performance as we cannot allocate a single task to single processor. However, at the time of writing the typical number of cores in a GPU was already $\sim$10k, and therefore the above analysis can be expected to become more and more accurate in the future with the steadily increasing number of computational cores. 

\subsection{Block processing} \label{sec:block-processing}

Up to now, we have considered parallelisation of control problems at a single time step level. That is, we initialise the element $a_k$ for all $k$ using \eqref{eq:gen-init} and then apply the parallel scan algorithm. Another option is to apply the parallel scan algorithm to non-overlapping blocks of $B$ time steps. That is, we can initialise the elements of the parallel scan with $V_{k \to k+B}$ for $k=S+nB$, with $n=0,1,\ldots,T/B-1$. The initialisation of each element can be done using \eqref{eq:vki-comb} in $B-1$ sequential steps. Then, we can apply the parallel scan algorithm to fuse the information from all blocks. In this case, the time-complexity in a PRAM computer with large enough number of processors is $O(B + \log (N/B))$, so it is optimal to choose $B=1$. Nevertheless, this approach can be useful if we have a limited number of processors.

In the case of LQT, an efficient implementation of the above can achieved by using partial condensing \cite{Axehill15,Frison16}, where the idea is to reformulate the problem in terms of blocks of states and controls of size $B$:
\begin{equation}
\begin{split}
  \bar{x}_{k/B} = \begin{bmatrix}
    x_{k} \\
    \vdots \\
    x_{k+B-1}
  \end{bmatrix}, \quad
  \bar{u}_{k/B} = \begin{bmatrix}
    u_{k} \\
    \vdots \\
    u_{k+B-1}
  \end{bmatrix}.
\end{split}
\end{equation}
By elimination of the state variables inside each block, we can reformulate the problem in terms of the initial states of blocks only, $x_0, x_{B}, x_{2B},\ldots,$ which reduces the state dimensions from full state blocks to the original state dimension. The resulting problem is still a LQT problem, but with modified states and inputs, and hence we can use the proposed parallel LQT algorithms to solve it. 

\section{Experimental Results} \label{sec:experimental}

In this section, we experimentally evaluate the performance of the methods in simulated applications. We implemented the methods using the open-source TensorFlow 2.6 software library \cite{Abadi_et_al:2015} using its Python 3.8 interface, which provides means to run parallel vectorised operations and parallel associative scans on GPUs. The experiments were run using NVIDIA A100-SXM GPU with 80GB of memory. In the experiments, we concentrate on the computational speed benefits because the sequential and parallel version of the algorithm compute exactly the same solution (provided that it is unique), the only difference being in the computational speed. The computation speeds were measured by averaging over 10 runs and the time taken to (jit) compile the code was not included in the measurement. 

We would like to point out that, as we are using TensorFlow with GPUs, the matrix operations on the individual time steps of the sequential algorithms are parallelised. Therefore, the sequential LQT algorithms can be interpreted as a TensorFlow parallel implementation of the Riccati recursion \cite{Frison13}.

\subsection{Experiment with basic LQT} \label{sec:basic_lqt_exp}

\begin{figure}[tb]
\centering
\includegraphics{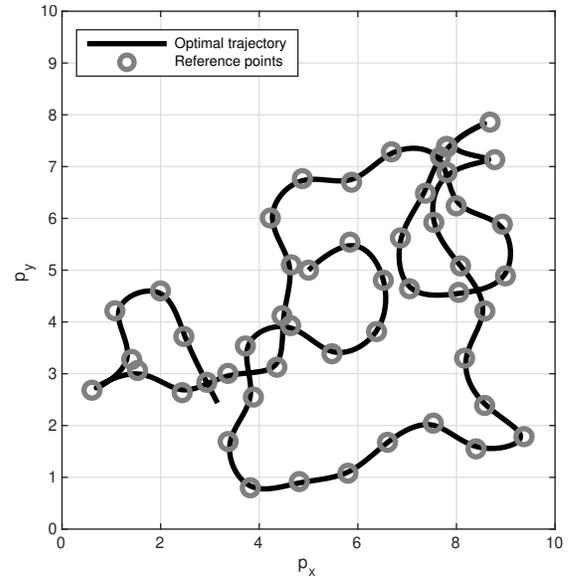}
\caption{Simulated trajectory from the linear control problem and optimal trajectory produced by LQT (see Section \ref{sec:basic_lqt_exp}). The trajectory starts at $(5,5)$.}
\label{fig:lqt_trajectory}
\end{figure}

The aim of the experiment is to demonstrate the benefit of the proposed parallelization method over the classical sequential solution in an LQT problem. We consider a 2-D tracking problem obeying Newton's law \cite[Example 4.4.2]{Lewis+Syrmos:1995}. In this LQT problem, the aim is to steer an object to follow a given trajectory of reference points in 2-D by using applied forces as the control signals. 

The state consists of the positions and velocities $x = \begin{bmatrix} p_x & p_y & v_x & v_y \end{bmatrix}^\top$ and the control signal $u = \begin{bmatrix} a_x & a_y \end{bmatrix}^\top$ contains the accelerations (forces divided by the mass which is unity in our case). If we assume that the control signal is kept fixed over each discretisation interval $[t_k,t_{k+1}]$, then the dynamic model can be written as
\begin{equation}
  x_{k+1} = F_k \, x_k + L_k \, u_k,
\end{equation}
where
\begin{equation}
  F_k = \begin{bmatrix} 1 & 0 & \Delta t_k &  0 \\ 0 & 1 &  0 & \Delta t_k \\ 0 & 0 &  1 &  0 \\ 0 & 0 & 0 & 1 \end{bmatrix}, 
  L_k = \begin{bmatrix}  \Delta t_k^2/2 & 0 \\ 0 & \Delta t_k^2/2 \\ \Delta t_k & 0 \\  0 & \Delta t_k \end{bmatrix},
\end{equation}
and $\Delta t_k = t_{k+1} - t_k$ is the sampling interval. Fig.~\ref{fig:lqt_trajectory} illustrates the scenario.

\begin{figure}[tb]
\centering
\includegraphics[width=0.49\columnwidth]{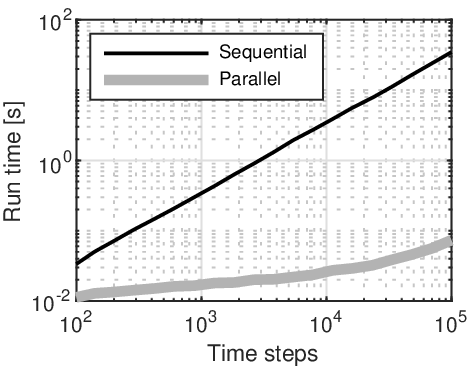}
\includegraphics[width=0.49\columnwidth]{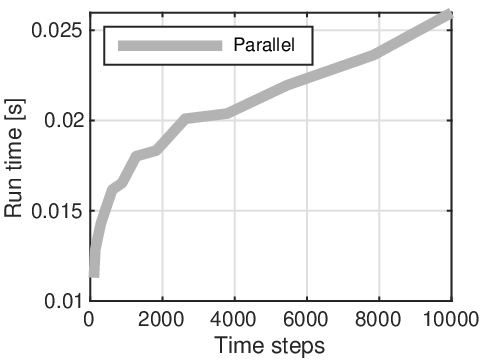}
\includegraphics[width=0.9\columnwidth]{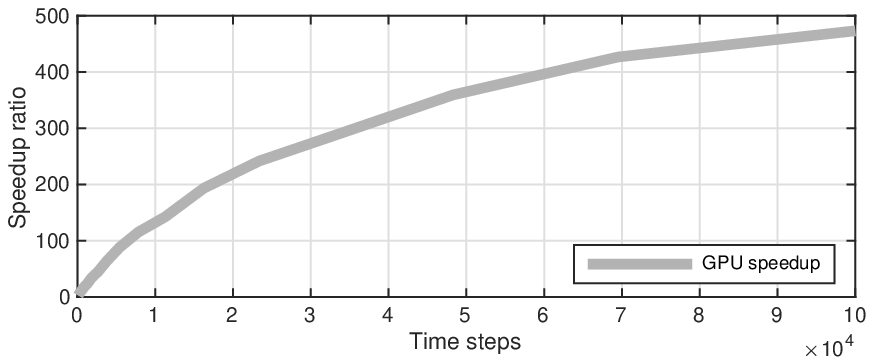}
\caption{LQT control law computation run times on GPU. The sequential and parallel run times are shown in the top left figure, and a zoom to the parallel run time is shown in the top right figure. The speed-up provided by the parallelisation is shown in figure at the bottom.}
\label{fig:track_backward_gpu}
\end{figure}

\begin{figure}[tb]
\centering
\includegraphics[width=0.49\columnwidth]{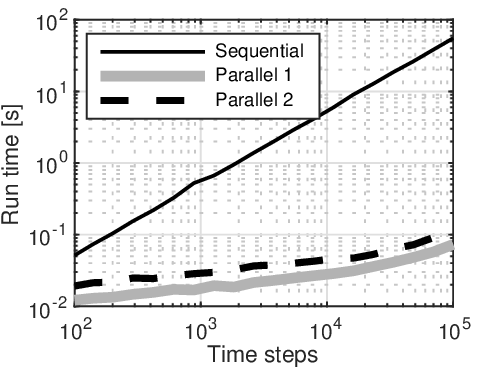}
\includegraphics[width=0.49\columnwidth]{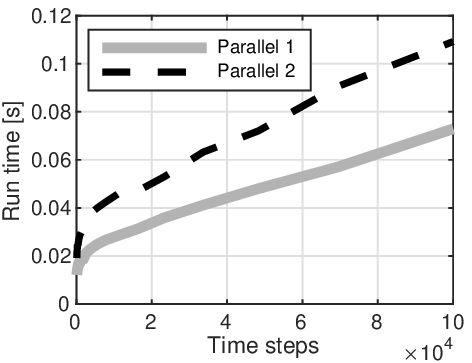}
\caption{The GPU run times (left) and zoomed run times of the parallel methods (right) for combined control law computation and trajectory recovery (Methods 1 \& 2) in LQT.}
\label{fig:track_backward-forward_gpu}
\end{figure}

The dynamic trajectory is discretized so that we add 10 intermediate steps between each of the reference point time steps which then results in a total of $T$ times steps (giving $\Delta t_k = 0.1$). The cost function parameters are for $k=0,\ldots,T-1$ selected to be
\begin{equation}
  H_k = \begin{bmatrix}
    1 & 0 & 0 & 0 \\
    0 & 1 & 0 & 0
  \end{bmatrix}, X_k = \begin{bmatrix}
    c_k & 0 \\
    0 & c_k
  \end{bmatrix}, 
  U_k = 10^{-1} \, I_{2 \times 2},
\end{equation}
where $c_k = 100$ when there is a reference point at step $k$ and $10^{-6}$ otherwise. At the final step we set $H_T = I_{4 \times 4}$ and $X_T = I_{4 \times 4}$. The reference trajectory contains the actual reference points at every 10th step $k$, and the intermediate steps are set equal to the previous reference point. At the final step the reference velocity is also zero. Together with the value $c_k = 10^{-6}$ at these intermediate points this results in tiny regularisation of the intermediate paths, but the effect on the final result is practically negligible. It would also be possible to put $c_k = 0$ for the intermediate steps to yield almost the same result.

The results computing the control law (i.e, the backward pass) of the classic sequential LQT and the proposed parallel LQT on the GPU for $T = 10^2,\ldots,10^5$ are shown in Fig.~\ref{fig:track_backward_gpu}. The figure shows the run times of both in log-log scale on the top right figure and the parallel result up to $10^4$ is shown in linear scale on the top right. The speed-up, computed as the ratio of sequential and parallel run times, is shown in the bottom figure. It can be seen that the parallel version is significantly faster than the sequential version (illustrated in the top left figure) and the logarithmic scaling of the parallel algorithm can also be seen (illustrated in the top right figure). The speed-up (illustrated in the bottom figure) is of the order of $\sim$470 with $10^5$ time points and although it is close to saturating, it still increases a bit.

We also ran the combined control law computation pass and the trajectory recovery pass on GPU, and the results are shown in Fig.~\ref{fig:track_backward-forward_gpu}. The advantage of parallel version over the sequential version can again be clearly seen. Furthermore, in this case, Method 1 for the parallel trajectory recovery is faster than Method 2.

\begin{figure}[tb]
\centering
\includegraphics[width=0.9\columnwidth]{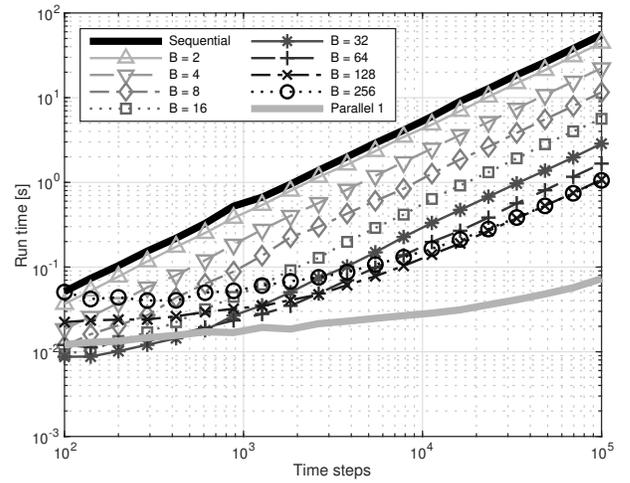}
\caption{Result of partial condensing with different values of block size $B$ when the LQT is implemented by using parallel matrix operations within sequential Riccati solution and trajectory reconstruction.}
\label{fig:track_gpu_seq_cond_time}
\end{figure}

\begin{figure}[htb]
\centering
\includegraphics[width=0.9\columnwidth]{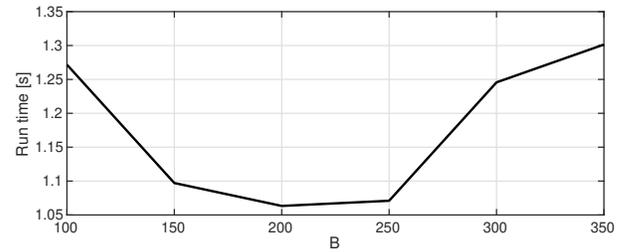}
\caption{Run times of partial condensing for $T = 10^5$ with different values of block size $B$ when the LQT is implemented by using parallel matrix operations within sequential Riccati solution and trajectory reconstruction.}
\label{fig:track_gpu_seq_cond_nc}
\end{figure}

\subsection{Experiment with partial condensing}
\textcolor{black}{} %
In this experiment we compare the proposed parallelisation method for LQT to partial condensing \cite{Axehill15,Frison16} based parallelisation. We also demonstrate how parallel condensing can be combined with the proposed methodology to yield improved results. The same 2-D tracking problem and data are used as in Sec.~\ref{sec:basic_lqt_exp}. 

As discussed in Sec.~\ref{sec:block-processing}, the idea of partial condensing is to reduce a control problem of length $T$ to a control problem of length $T/B$ by dividing the problem into blocks of length $B$. In each of these blocks we can eliminate all but the state at the beginning of the block, which effectively reduces a control problem with state dimension $n_x$, input dimension $B \, n_u$, and length $T/B$. The required computations are 1) conversion of the model into block form, 2) computation of LQT solution of length $T/B$, and 3) reconstruction of the intermediate states in each block. Partial condensing allows for parallelisation of the computations because steps 1) and 3) are fully parallelisable and step 2) can be efficiently implemented by using parallel matrix operations within the sequential Riccati solution and trajectory reconstruction \cite{Axehill15,Frison16,Frison13}.

The GPU run times of the aforementioned parallel condensing method with $B = 2,4,8,\ldots,256$ are shown Fig.~\ref{fig:track_gpu_seq_cond_time}. The run times of the classical backward-forward sequential LQT solution for trajectory recovery and of the proposed parallel method with trajectory recovery with Method 1 are also shown in the figure. The trajectory lengths were $T = 10^2,\ldots,10^5$. It can be seen that the run times of partial condensing methods are significantly lower than of the classical sequential LQT while still, for the most of the cases, higher than of the proposed parallel method. It can be seen that with short trajectory lengths the partial condensing method is faster than the proposed parallel method when $B = 16$ or $32$. However, with larger trajectory lengths the proposed parallel method is faster.

In the results of Fig.~\ref{fig:track_gpu_seq_cond_time} we can see some evidence of an effect that when increasing $B$, the run time no longer decreases after a certain value. This is confirmed in Fig.~\ref{fig:track_gpu_seq_cond_nc} which shows the run times of the partial condensing over trajectory of length $T = 10^5$ with different values of $B$. It can be seen that the run time attains minimum somewhere around $B = 200$ and after that the run time starts to increase.

\begin{figure}[tb]
\centering
\includegraphics[width=0.9\columnwidth]{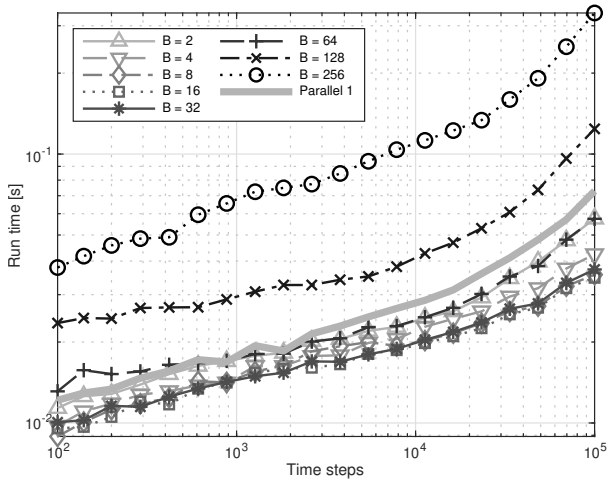}
\caption{Result of partial condensing with different values of $N_c$ when the LQT is implemented by using the proposed parallel methods.}
\label{fig:track_gpu_par_cond_time}
\end{figure}

\begin{figure}[htb]
\centering
\includegraphics[width=0.9\columnwidth]{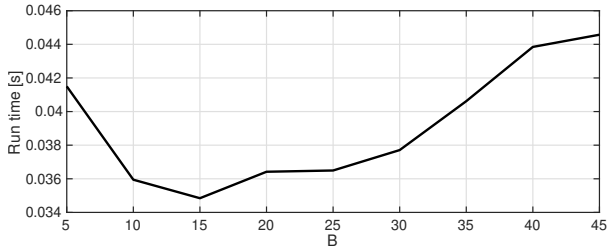}
\caption{Run times of partial condensing for $T = 10^5$ with different values of $B$ when the LQT is implemented by using the proposed parallel methods.}
\label{fig:track_gpu_par_cond_nc}
\end{figure}

As discussed in Sec.~\ref{sec:block-processing}, it is also possible to combine partial condensing with the proposed parallelization methodology. This can be done by implementing the LQT solution (of length $T/B$) by using one of the parallel methods. Fig.~\ref{fig:track_gpu_par_cond_time} shows the results of using this combined approach. It can be seen that the partial condensing can be used to improve run time of the proposed parallel methods when $B$ is in suitable range (2--64 in this case). When $B$ is too large ($128$ or $256$), then partial condensing no longer improves the run times.  This also happens with $B = 64$ when the trajectory length is short. Fig.~\ref{fig:track_gpu_par_cond_nc} shows the run times with $T = 10^5$ as function of $B$. It can be seen in the figure that the minimum run time is attained roughly at value $B = 15$.

\textcolor{black}{} %
\subsection{Experiment with increasing state dimensionality} \label{sec:spring_exp}
In this experiment the aim is to test the scaling of run time when the dimensionality of the state increases. For this purpose, we use a slight modification of a mass-spring-damper problem \cite{Englert:2019}\cite[Section 2.5.3]{kapernick2016gradient}, where we have removed the control constraints, as we do not consider them in this paper. This is a linear model for controlling a chain of $N$ masses $m = 1$ kg connected with springs with constants $c = 1$ kg/s$^2$ and dampers with constants $d = 0.2$ kg/s. The control is applied to the first and last mass. The model is thus (see Fig.~\ref{fig:masses}):
\begin{equation}
\begin{split}
  \ddot{y}_1 &= \frac{c}{m} \left[-2 y_1 + y_2\right] + \frac{d}{m} \left[-2 \dot{y}_1 + \dot{y}_2 \right] + \frac{1}{m} \, u_1, \\
  \ddot{y}_i &= \frac{c}{m} \left[y_{i-1} - 2 y_i + y_{i+1} \right] + \frac{d}{m} \left[ \dot{y}_{i+1} - 2 \dot{y}_i + \dot{y}_{i+1} \right], \\
  \ddot{y}_N &= \frac{c}{m} \left[-2 y_N + y_{N-1} \right] + \frac{d}{m} \left[-2 \dot{y}_N + \dot{y}_{N-1} \right] - \frac{1}{m} \,  u_2.
\end{split}
\end{equation}
where $i=2,\ldots,N-1$. The state of the system is $x = \begin{bmatrix} y_1 & \dot{y}_1 & \cdots & y_N & \dot{y}_N \end{bmatrix}^\top$. Similar to the case in \cite{kapernick2016gradient}, the aim is to control the system to origin from an initial condition where the first mass and middle mass, with index $i=\left\lfloor \frac{N}{2}\right\rfloor + 1 $, are started at position $1$ m. The model is uniformly discretised, with closed-form zero-order-hold (ZOH) discretisation, using varying number of time steps $T = 10^2,\ldots,10^3$ such that the total control interval length is $10$ s.

\begin{figure}[tb]
	\centering
	\includegraphics[width=0.9\columnwidth]{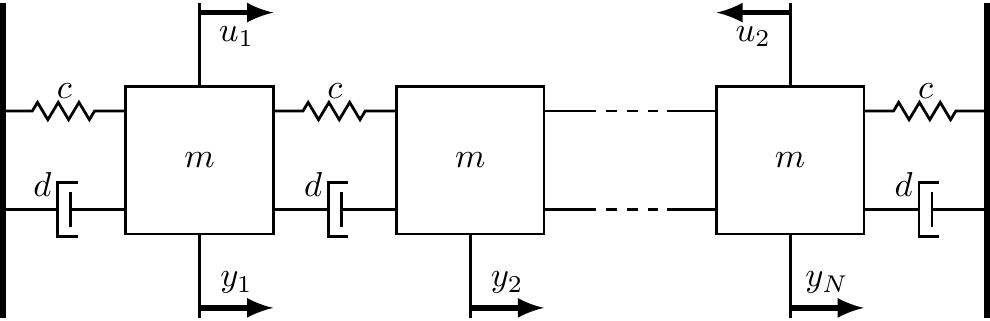}
	\caption{Illustration of the mass-spring-damper problem (see Section \ref{sec:spring_exp}).}
	\label{fig:masses}
\end{figure}

The cost function is
\begin{equation}
\begin{split}
  C[u_{0:T}] &= \frac{1}{2} x^\top_T X_T x_T + \frac{1}{2} \sum_{n=0}^{T-1} x^\top_n X x_n 
  + \frac{1}{2} \sum_{n=0}^{T-1} u^\top_n U u_n,
\end{split}
\end{equation}
with $X = X_T = I$ and $U = 0.1 I$.

\begin{figure}[tb]
\centering
\includegraphics[width=0.9\columnwidth]{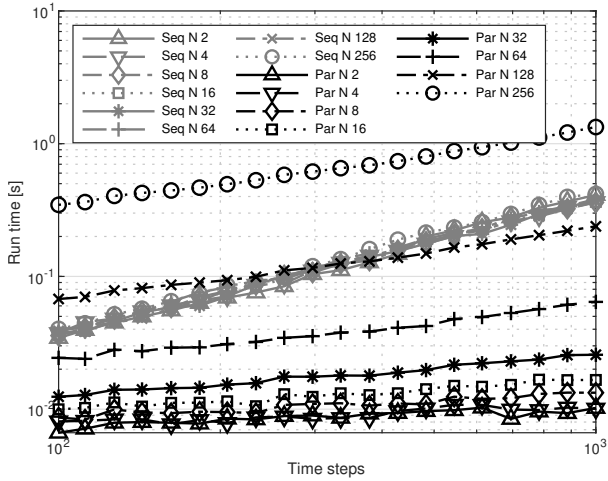}
\caption{Run times of the mass-spring-damper problem with different number of masses $N$ corresponding to state dimensionalities $2N$.}
\label{fig:mass_2}
\end{figure}

Fig.~\ref{fig:mass_2} shows the results of sequential LQT and proposed parallel LQT. Due to parallelisation of the matrix operations on the individual steps of the sequential LQT, its run time is essentially independent of the state dimension. The parallel method, however, experiences significant run time increase with larger state dimension. Although when the number of masses $N$ is 2--64, the run times of the parallel methods are shorter than those of the sequential method, with $N=128$ the parallel method is slower with small numbers of time steps and with $N=256$ it is slower with all the time step counts. 

\subsection{Experiment with finite state space}\label{sec:finite_exp}
\textcolor{black}{} %
In this experiment, we consider an aircraft routing problem \cite[Example 6.1.1]{Lewis+Syrmos:1995}, where an aircraft  proceeds from left to right, and the aim is to control up and down movement on a finite grid so that the total cost is minimised. Each of the grid points incurs a cost $ \{ 0,1,2 \}$ which is related to the fuel required to go through it. Taking a control up or down costs a single unit, and proceeding straight costs nothing. The scenario is illustrated in Fig.~\ref{fig:fsc_track}. 

\begin{figure}[tb]
\centering
\includegraphics{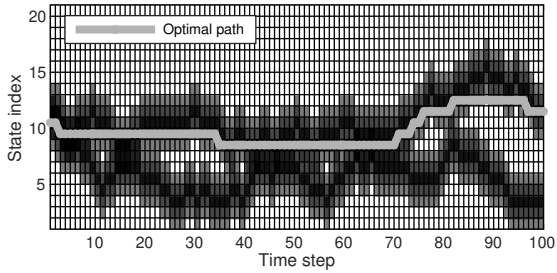}
\caption{Finite state space scenario where the aim is to find a minimum cost path from left to right by steering up or down (see Section \ref{sec:finite_exp}). The gray scale values show the cost function values.}
\label{fig:fsc_track}
\end{figure}

In this case we tested the finite-state control law computation using different state dimensionalities as the parallel combination rule can be expected to have a dependence on the state dimensionality when the number of computational cores is limited. The GPU speed-ups for state dimensions $D_x \in \{ 5, 11, 21 \}$ are shown in Fig.~\ref{fig:fsc_backward_gpu_speedups}. It can be seen that with state dimensionality $D_x = 5$ the speed-up reaches $\sim$1500 with $T = 10^5$ and is still slightly increasing. With the state dimensionality $D_x = 11$ the maximum achieved speed-up is roughly $\sim$700, and with state dimensionality $D_x = 21$ the speed-up saturates to a value around $350$. However, with all of the state dimensionalities parallelisation provides a significant speed-up.

\begin{figure}[tb]
\centering
\includegraphics[width=0.9\columnwidth]{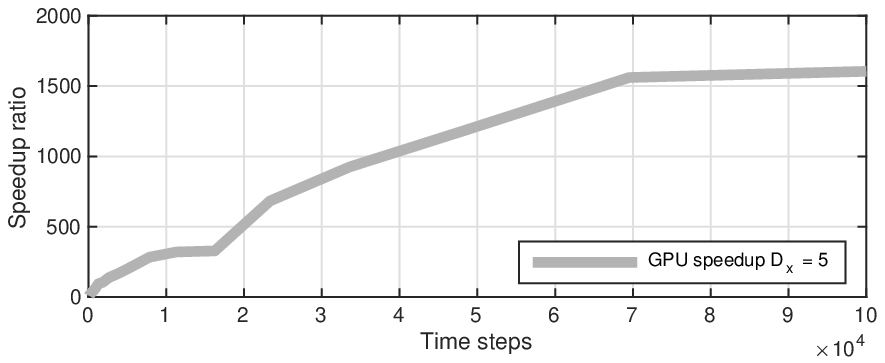}
\includegraphics[width=0.9\columnwidth]{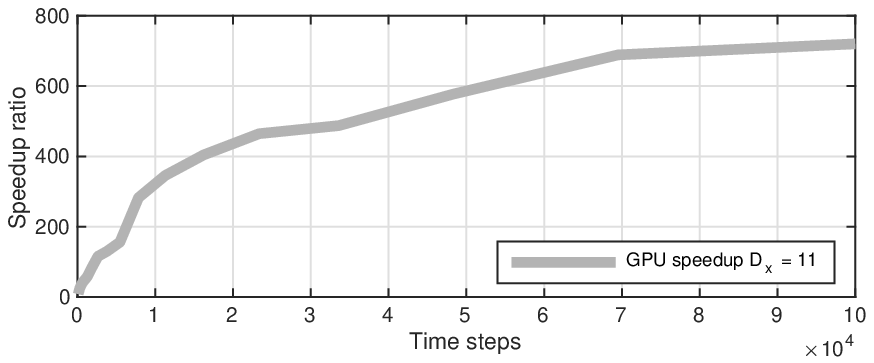}
\includegraphics[width=0.9\columnwidth]{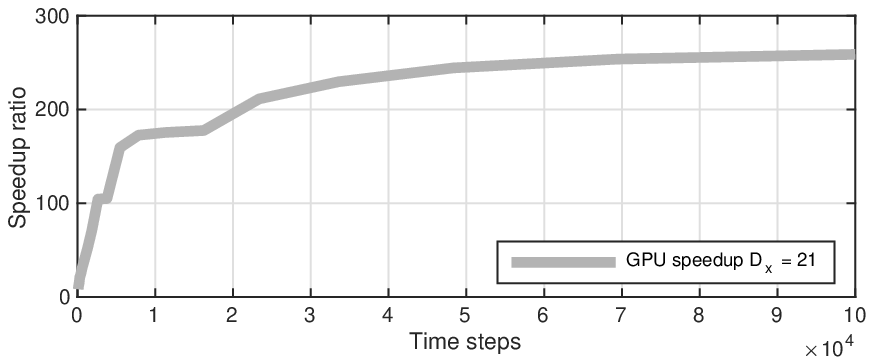}
\caption{Finite state space GPU speed-ups for control law computation (parallel vs. sequential) with state dimensions 5, 11, and 21.}
\label{fig:fsc_backward_gpu_speedups}
\end{figure}

\subsection{Experiment with nonlinear LQT}\label{sec:non_linear_LQT_exp}

This experiment is concerned with a non-linear dynamic model where we control a simple unicycle \cite[Sec. 13.2.4.1]{LaValle_book06} whose state consists of 2-D position, orientation, and speed $x = \begin{bmatrix} p_x & p_y & \theta & s  \end{bmatrix}^\top$. The aim is to steer the device to follow a given position and orientation trajectory which corresponds to going around a fixed race track multiple times. The control signal consists of the tangential acceleration and turn rate $u = \begin{bmatrix} a & \omega \end{bmatrix}^\top$. The discretized nonlinear model has the form
\begin{equation}
  x_{k+1} = f_k(x_k,u_k),
\end{equation}
where
\begin{equation}
  f_k(x,u) = \begin{bmatrix} p_x + s \, \cos(\theta) \, \Delta t_k \\
                          p_y + s \, \sin(\theta) \, \Delta t_k \\
                          \theta + \omega \, \Delta t_k \\
                          s + a \, \Delta t_k
    \end{bmatrix},
\end{equation}
and $\Delta t_k = t_{k+1} - t_k$. Fig.~\ref{fig:nlqt_trajectory} shows the trajectory and the optimal trajectory produced by the nonlinear LQT. 

\begin{figure}[tb]
\centering
\includegraphics{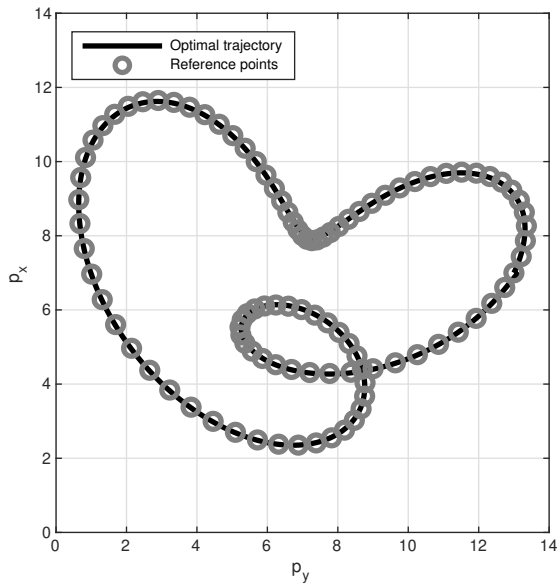}
\caption{Simulated trajectory from the nonlinear control problem and optimal trajectory produced by nonlinear LQT (see Section \ref{sec:non_linear_LQT_exp}).}
\label{fig:nlqt_trajectory}
\end{figure}

The cost function parameters were selected to be the following for $k=0,\ldots,T-1$:
\begin{equation}
  H_k = \begin{bmatrix}
    1 & 0 & 0 & 0 \\
    0 & 1 & 0 & 0 \\
    0 & 0 & 1 & 0
  \end{bmatrix}, X_k = \begin{bmatrix}
    c_k & 0 & 0 \\
    0   & c_k  & 0 \\
    0   &   0   & d_k
  \end{bmatrix}, 
  U_k = \begin{bmatrix}
    1 & 0 \\
    0 & 100
  \end{bmatrix},
\end{equation}
where $c_k = 100, d_k = 1000$, when there is a reference point at step $k$, and $10^{-6}$ otherwise. The latter values were also used for the terminal step $k = T$. The time step length was $\Delta t_k = 0.1$.

\begin{figure}[tb]
\centering
\includegraphics[width=0.49\columnwidth]{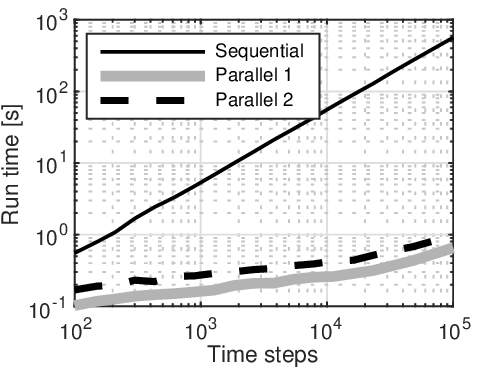}
\includegraphics[width=0.49\columnwidth]{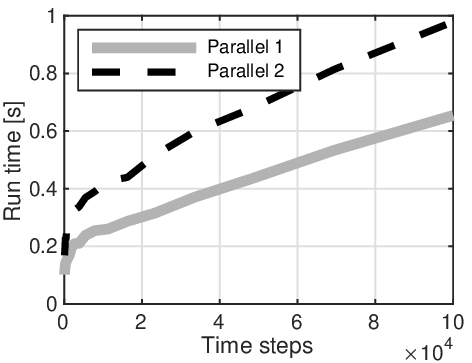}
\includegraphics[width=0.9\columnwidth]{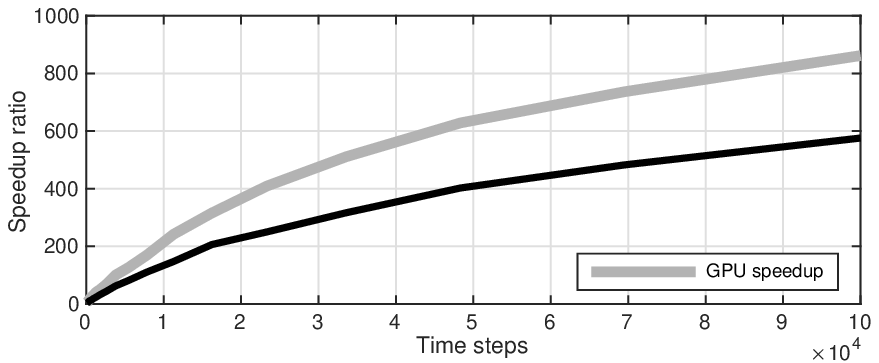}
\caption{Nonlinear LQT GPU run times and speedup for 10 iterations.}
\label{fig:nlqt_gpu}
\end{figure}

An iterated nonlinear LQT using a Taylor series approximation was applied to the model, and the number of iterations was fixed to $10$. Fig.~\ref{fig:nlqt_gpu} shows the run times for GPU. It can be seen that parallelisation provides a significant speed-up over sequential computation. When the Method~1 was used to compute the recovered trajectory at each iteration step, the speed-up grows to around $800$ for $T=10^5$ on GPU. Method~2 reaches a speed-up of around $600$.

\section{Conclusion} \label{sec:conclusion}
In this paper, we have shown how dynamic programming solutions to optimal control problems and their linear quadratic special case, the linear quadratic tracker (LQT), can be parallelised in the temporal domain by defining the corresponding associative operators and making use of parallel scans. The parallel methods have logarithmic complexity with respect to the number of time steps, which significantly reduces the linear complexity of standard (sequential) methods for long time horizon control problems. These benefits are shown via numerical experiments run on a GPU. This paper shows that the contribution is timely as it can leverage modern hardware and software for parallel computing, such as GPUs and TensorFlow.

An interesting future extension of the framework would be parallel stochastic dynamic programming solution to stochastic control problems \cite{Stengel_book94,Maybeck:1982b}. As discussed in Section~\ref{sec:extensions}, this is straightforward in the LQT case due to certainty equivalence, but the general stochastic case is not straightforward. Formally it is possible to replace the state $x_k$ with the distribution of the state $p_k$ and consider condition value functionals of the form $V_{i \to j}[p_i,p_j]$ and value functionals of the form $V_{i}[p_i]$. The present framework then, in principle, applies as such. In particular, when the distributions have finite-dimensional sufficient statistics, this can lead to tractable methods.  Unfortunately, unlike in the sequential dynamic programming case, more generally, this approach does not seem to lead to a tractable algorithm.

Another interesting extension is to consider continuous optimal control problems in which case also the stochastic control solution has certain group properties \cite{Nisio:2015} which might allow for parallelisation. However, the benefit of parallelisation in the continuous case is not as clear as in discrete-time case because of the infinite number of time steps.

\appendices

\section{Derivations and Proofs} \label{sec:Appendix}

\subsection{Proof of Theorem~\ref{the:v-comb}}

In this appendix we prove Theorem~\ref{the:v-comb}. We first prove \eqref{eq:vk-comb}. From \eqref{eq:det_problem} and \eqref{eq:Value_function}, we obtain
\begin{equation}
 \begin{split} & V_{k}(x_{k})\\
 & =\min_{u_{k:T-1}}\ell_{T}(x_{T})+\sum_{n=k}^{T-1}\ell_{n}(x_{n},u_{n})\\
 & =\min_{u_{k:T-1}}\sum_{n=k}^{i-1}\ell_{n}(x_{n},u_{n})+\sum_{n=i}^{T-1}\ell_{n}(x_{n},u_{n})+\ell_{T}(x_{T})\\
 & =\min_{u_{k:i-1}}\sum_{n=k}^{i-1}\ell_{n}(x_{n},u_{n})+\min_{u_{i:T-1}}\left[\sum_{n=i}^{T-1}\ell_{n}(x_{n},u_{n})+\ell_{T}(x_{T})\right]\\
 & =\min_{u_{k:i-1}}\sum_{n=k}^{i-1}\ell_{n}(x_{n},u_{n})+V_{i}(x_{i}),
 \end{split}
\end{equation}
where the previous minimisations are subject to the trajectory constraints in \eqref{eq:det_problem}.

We can also minimise over $x_i$ explicitly such that
\begin{align}
V_{k}(x_{k})&=\min_{x_{i}}\left[\min_{u_{k:i-1}}\left[\sum_{n=k}^{i-1}\ell_{n}(x_{n},u_{n})\right]+V_{i}(x_{i})\right]\\
&=\min_{u_{i-1}}\left[V_{k\to i}(x_{k},x_{i})+V_{i}(x_{i})\right],
\end{align}
which proves \eqref{eq:vk-comb}.

Proceeding analogously, we now prove \eqref{eq:vki-comb}. From \eqref{eq:vki-definition}, we obtain
\begin{align}
&V_{k\to i}(x_{k},x_{i}) \nonumber\\
&=\min_{u_{k:i-1}}\sum_{n=k}^{i-1}\ell_{n}(x_{n},u_{n}) \nonumber\\
&=\min_{u_{k:j-1}}\left[\sum_{n=k}^{j-1}L_{n}(x_{n},u_{n})+\min_{u_{j:i-1}}\sum_{n=j}^{i-1}\ell_{n}(x_{n},u_{n})\right] \nonumber\\
&=\min_{u_{k:j-1}}\left[\sum_{n=k}^{j-1}\ell_{n}(x_{n},u_{n})+V_{j\to i}(x_{j},x_{i})\right] \nonumber\\
&=\min_{x_{j}}\min_{u_{k:j-1}}\left[\sum_{n=k}^{j-1}\ell_{n}(x_{n},u_{n})+V_{j\to i}(x_{j},x_{i})\right] \nonumber\\
&=\min_{x_{j}}\left[\min_{u_{k:j-1}}\left(\sum_{n=k}^{j-1}\ell_{n}(x_{n},u_{n})\right)+V_{j\to i}(x_{j},x_{i})\right] \nonumber\\
&=\min_{x_{j}}\left[V_{k\to j}(x_{k},x_{j})+V_{j\to i}(x_{j},x_{i})\right],
\end{align}
which completes the proof of \eqref{eq:vki-definition}.

\subsection{Proof of LQT combination rule}\label{sec:LQT_combination_append}
In this appendix, we prove the combination rule for LQT in Lemma \ref{lem:LQT_combination}. 
Combining $V_{k\to j}(x_{k},x_{j})$ and $V_{j\to i}(x_{j},x_{i})$ of the form \eqref{eq:Vki_dual}, we obtain
\begin{align*}
& V_{k\to i}(x_{k},x_{i})\\
& =\min_{x_{j}}\left\{ \max_{\lambda_{1}}g_{k\to j}(\lambda_{1};x_{k},x_{j})+\max_{\lambda_{2}}g_{j\to i}(\lambda_{2};x_{j},x_{i})\right\} \\
& =\mathrm{z}+\max_{\lambda_{1},\lambda_{2}}\min_{x_{j}}\left\{ \frac{1}{2}x_{k}^{\top}J_{k,j}x_{k}-x_{k}^{\top}\eta_{k,j}-\frac{1}{2}\lambda_{1}^{\top}C_{k,j}\lambda_{1}\right.\\
& -\lambda_{1}^{\top}\left(x_{j}-A_{k,j}x_{k}-b_{k,j}\right)+\frac{1}{2}x_{j}^{\top}J_{j,i}x_{j}-x_{j}^{\top}\eta_{j,i}\\
& \left.-\frac{1}{2}\lambda_{2}^{\top}C_{j,i}\lambda_{2}-\lambda_{2}^{\top}\left(x_{i}-A_{j,i}x_{j}-b_{j,i}\right)\right\}.
\end{align*}
We prove the result by calculating the minimum w.r.t. $x_{j}$ and
maximum w.r.t. $\lambda_{1}$, leaving the Lagrange multiplier $\lambda_{2}$
as the Lagrange multiplier of $V_{k\to i}(x_{k},x_{i})$.

Setting the gradient w.r.t. $x_{j}$ equal to zero, we obtain
\begin{align}
J_{j,i}x_{j} & =\lambda_{1}+\eta_{j,i}-A_{j,i}^{\top}\lambda_{2}\label{eq:derivative_x_j},
\end{align}
where $J_{j,i}$ is not invertible in general.

Setting the gradient w.r.t $\lambda_{1}$ equal to zero, we have
\begin{align}
x_{j} & =-C_{k,j}\lambda_{1}+A_{k,j}x_{k}+b_{k,j}\label{eq:x_j_prev}.
\end{align}
Then, substituting (\ref{eq:x_j_prev}) into (\ref{eq:derivative_x_j})
yields
\begin{align}
\lambda_{1} & =\left(I+J_{j,i}C_{k,j}\right)^{-1}\left[-\eta_{j,i}+A_{j,i}^{\top}\lambda_{2}+J_{j,i}\left(A_{k,j}x_{k}+b_{k,j}\right)\right].\label{eq:labmda_1}
\end{align}
Substituting (\ref{eq:labmda_1}) into (\ref{eq:x_j_prev}), we obtain
\begin{align}
x_{j} & =-C_{k,j}\left(I+J_{j,i}C_{k,j}\right)^{-1}A_{j,i}^{\top}\lambda_{2}\nonumber \\
& -C_{k,j}\left(I+J_{j,i}C_{k,j}\right)^{-1}\left[-\eta_{j,i}+J_{j,i}\left(A_{k,j}x_{k}+b_{k,j}\right)\right]\nonumber \\
& +A_{k,j}x_{k}+b_{k,j}\label{eq:x_j}.
\end{align}

We now substitute the stationary points (\ref{eq:labmda_1}) and (\ref{eq:x_j})
in each of the terms in $V_{k\to i}(x_{k},x_{i})$ to recover a function of the form \eqref{eq:Vki_dual}. This step involves the use of long mathematical expressions so it is left out of the paper. Lemma \ref{lem:LQT_combination} then follows by term identification. 

\subsection{Proof of Parallel LQT} \label{sec:Proof_parallel_LQT_append}
In this section, we prove Lemma \ref{lem:LQT_assoc}. 

\subsubsection{Proof of \eqref{eq:LQT_lem_V_S_k}}
We first show the form of $V_{k\to k+1}(x_{k},x_{k+1})$ in its dual
representation for $k=S,\ldots,T$. Using \eqref{eq:vki-definition} and the LQT problem formulation in \eqref{eq:LQT_problem}, we obtain
\begin{align}
& V_{k\to k+1}(x_{k},x_{k+1})\nonumber \\
& =\min_{u_{k}}\ell_{k}(x_{k},u_{k})\nonumber \\
& =\min_{u_{k}}\frac{1}{2}(H_{k}x_{k}-r_{k})^{\top}X_{k}(H_{k}x_{k}-r_{k})+\frac{1}{2}u_{k}^{\top}U_{k}u_{k}\nonumber \\
& =\mathrm{z}+\min_{u_{k}}\frac{1}{2}x_{k}^{\top}H_{k}^{\top}X_{k}H_{k}x_{k}-x_{k}^{\top}H_{k}^{\top}X_{k}r_{k}+\frac{1}{2}u_{k}^{\top}U_{k}u_{k}
\end{align}
subject to 
\begin{align}
x_{k+1} & =F_{k}x_{k}+c_{k}+L_{k}u_{k}\label{eq:LQT_parallel_proof_constraint-1}.
\end{align}

The Lagrangian of $V_{k\to k+1}(x_{k},x_{k+1})$ is \cite{Boyd_book04} 
\begin{align*}
& L_{k\to k+1}\left(u_{k},\lambda;x_{k},x_{k+1}\right)\\
& =\mathrm{z}+\frac{1}{2}x_{k}^{\top}H_{k}^{\top}X_{k}H_{k}x_{k}-x_{k}^{\top}H_{k}^{\top}X_{k}r_{k}\\
& +\frac{1}{2}u_{k}^{\top}U_{k}u_{k}+\lambda^{\top}\left(L_{k}u_{k}-\left(x_{k+1}-F_{k}x_{k}-c_{k}\right)\right)
\end{align*}

and the dual function is \cite{Boyd_book04}
\begin{align}
& g_{k\to k+1}\left(\lambda;x_{k},x_{k+1}\right)\nonumber\\
& =\min_{u_{k}}L_{k\to k+1}\left(u_{k},\lambda;x_{k},x_{k+1}\right)\nonumber\\
& =\mathrm{z}+\frac{1}{2}x_{k}^{\top}H_{k}^{\top}X_{k}H_{k}x_{k}-x_{k}^{\top}H_{k}^{\top}X_{k}r_{k}\nonumber\\
& \quad-\frac{1}{2}\lambda^{\top}L_{k}U_{k}^{-1}L_{k}^{\top}\lambda-\lambda^{\top}\left(x_{k+1}-F_{k}x_{k}-c_{k}\right),\label{eq:g_k_1_append}
\end{align}
where the minimum is obtained setting the gradient of $L_{k\to k+1}\left(\cdot\right)$
w.r.t. $u_{k}$ equal to zero, which gives
\begin{align}
u_{k} & =-U_{k}^{-1}L_{k}\lambda.
\end{align}
Comparing \eqref{eq:g_k_1_append} with \eqref{eq:g_ki} proves the initialisation in \eqref{eq:lqt_init}. Then, by applying Theorem \ref{the:v-comb}, which is equivalent to Lemma \ref{lem:LQT_combination} in the LQT setting, we complete the proof of \eqref{eq:LQT_lem_V_S_k}.

\subsubsection{Proof of \eqref{eq:LQT_lem_V_k}}
We use induction backwards to prove \eqref{eq:LQT_lem_V_k}. At the last time step, from Lemma \ref{lem:LQT_assoc}, we have
\begin{align*}
& V_{T\to T+1}(x_{T},x_{T+1})\\
& =\max_{\lambda}\left[\mathrm{z}+\frac{1}{2}x_{T}^{\top}H_{T}^{\top}X_{T}H_{T}x_{T}-x_{T}^{\top}H_{T}^{\top}r_{T}-\lambda^{\top}x_{T+1}\right].
\end{align*}
For $x_{T+1}\neq0$, this function is infinite. For $x_{T+1}=0$,
we have
\begin{align*}
V_{T\to T+1}(x_{T},0) & =\frac{1}{2}x_{T}^{\top}H_{T}^{\top}X_{T}H_{T}x_{T}-x_{T}^{\top}H_{T}^{\top}r_{T}
\end{align*}
which coincides with $V_{T}(x_{T})$, see Section \ref{subsec:LQT}.

We now assume that \eqref{eq:LQT_lem_V_k} holds for $k+1$, which implies that we have
\begin{align}
A_{k+1,T+1}&=0,\nonumber\\
b_{k+1,T+1}&=0,\nonumber\\
C_{k+1,T+1}&=0,\nonumber\\
\eta_{k+1,T+1}&=v_{k+1},\nonumber\\
J_{k+1,T+1}&=S_{k+1}, \nonumber
\end{align}
where $v_{k+1}$ and $S_{k+1}$ are the parameters in \eqref{eq:V_k_LQT}, and then show that \eqref{eq:LQT_lem_V_k} holds for $k$. From Lemma \ref{lem:LQT_assoc}, we have
\begin{align}
A_{k,k+1}&=F_{k},\nonumber\\
b_{k,k+1}&=c_{k},\nonumber\\
C_{k,k+1}&=L_{k}U_{k}^{-1}L_{k}^{\top},\nonumber\\
\eta_{k,k+1}&=H_{k}^{\top}X_{k}r_{k},\nonumber\\
J_{k,k+1}&=H_{k}^{\top}X_{k}H_{k}.\nonumber
\end{align}
By applying the combination rules in Lemma \ref{lem:LQT_combination}, we obtain
\begin{align}
A_{k,T+1}&=0,\nonumber\\
b_{k,T+1}&=0,\nonumber\\
C_{k,T+1}&=0, \nonumber\\
\eta_{k,T+1}&=F_{k}^{\top}(I+S_{k+1}L_{k}U_{k}^{-1}L_{k}^{\top})^{-1}(v_{k+1}-S_{k+1}c_{k})\nonumber\\
&\quad+H_{k}^{\top}X_{k}r_{k},\nonumber\\
J_{k,T+1}&=F_{k}^{\top}(I+S_{k+1}L_{k}U_{k}^{-1}L_{k}^{\top})^{-1}S_{k+1}F_{k}+H_{k}^{\top}X_{k}H_{k}.
\end{align}
We need to prove that these equations are equivalent to \eqref{eq:v_k_recursion} and \eqref{eq:S_k_recursion}. We first prove that $\eta_{k,T+1}=v_{k}$, which requires proving that the following identity holds
\begin{align}
F_{k}^{\top}(I+S_{k+1}L_{k}U_{k}^{-1}L_{k}^{\top})^{-1}&=\left(F_{k}^{\top}-L_{k}K_{k}\right)^{\top}.
\end{align}
On one hand, the right hand side can be written as
\begin{align}
\left(F_{k}^{\top}-L_{k}K_{k}\right)^{\top}&=F_{k}-F_{k}^{\top}S_{k+1}L_{k}\left(L_{k}^{\top}S_{k+1}L_{k}+U_{k}\right)^{-1}L_{k}^{\top}.
\end{align}
On the other hand, by applying the matrix inversion lemma, the left-hand side becomes
\begin{align}
&F_{k}^{\top}(I+S_{k+1}L_{k}U_{k}^{-1}L_{k}^{\top})^{-1} \nonumber\\
&=F_{k}^{\top}\left(I-S_{k+1}L_{k}\left(U_{k}+L_{k}^{\top}S_{k+1}L_{k}\right)^{-1}L_{k}^{\top}\right), \label{eq:proof_parallel_LQT_mu}
\end{align}
which proves $\eta_{k,T+1}=v_{k}$.

To prove that $J_{k,T+1}=S_{k}$, we need to prove that 
\begin{align}
F_{k}^{\top}(I+S_{k+1}L_{k}U_{k}^{-1}L_{k}^{\top})^{-1}S_{k+1}F_{k}&=F_{k}^{\top}S_{k+1}(F_{k}-L_{k}K_{k}).
\end{align}
On one hand, the right-hand side can be written as
\begin{align}
&F_{k}^{\top}S_{k+1}(F_{k}-L_{k}K_{k}) \nonumber\\
&=F_{k}^{\top}S_{k+1}\left(F_{k}-L_{k}\left(L_{k}^{\top}S_{k+1}L_{k}+U_{k}\right)^{-1}L_{k}^{\top}S_{k+1}F_{k}\right) \nonumber\\
&=F_{k}^{\top}S_{k+1}F_{k}-F_{k}^{\top}S_{k+1}L_{k}\left(L_{k}^{\top}S_{k+1}L_{k}+U_{k}\right)^{-1}L_{k}^{\top}S_{k+1}F_{k}.
\end{align}
On the other hand, using \eqref{eq:proof_parallel_LQT_mu}, the left-hand side is
\begin{align}
&F_{k}^{\top}(I+S_{k+1}L_{k}U_{k}^{-1}L_{k}^{\top})^{-1}S_{k+1}F_{k}\nonumber\\
&=F_{k}^{\top}\left(I-S_{k+1}L_{k}\left(U_{k}+L_{k}^{\top}S_{k+1}L_{k}\right)^{-1}L_{k}^{\top}\right)S_{k+1}F_{k}\nonumber\\
&=F_{k}^{\top}S_{k+1}F_{k}-F_{k}^{\top}S_{k+1}L_{k}\left(U_{k}+L_{k}^{\top}S_{k+1}L_{k}\right)^{-1}L_{k}^{\top}S_{k+1}F_{k},
\end{align}
which proves the result.

\subsection{Proof of optimal trajectory recovery} \label{app:lqt_traj}
In this appendix, we prove Lemma \ref{lem:LQT_trajectory_method2}. Substituting $V_{k}(x_{k})$ of the form \eqref{eq:V_k_LQT} and $V_{S\to k}(x_{S},x_{k})$  of the form \eqref{eq:Vki_dual} into \eqref{eq:trajectory_alternative}, we obtain 
\begin{align}
x_{k}^{*} & =\arg\min_{x_{k}}\max_{\lambda}\frac{1}{2}x_{S}^{\top}J_{S,k}x_{S}-x_{S}^{\top}\eta_{S,k}\nonumber \\
& \quad-\frac{1}{2}\lambda^{\top}C_{S,k}\lambda-\lambda^{\top}\left(x_{k}-A_{S,k}x_{S}-b_{S,k}\right)\nonumber \\
& \quad+\frac{1}{2}x_{k}^{\top}S_{k}x_{k}-v_{k}^{\top}x_{k}.\label{eq:x_k_opt_append}
\end{align}

The minimum of \eqref{eq:x_k_opt_append} w.r.t. $x_{k}$ can be found by setting the gradient of
the function equal to zero, which yields
\begin{align}
x_{k} & =S_{k}^{-1}\left(\lambda+v_{k}\right).\label{eq:x_k_trajectory_prev}
\end{align}

We substitute (\ref{eq:x_k_trajectory_prev}) into the function (without
argmin) in (\ref{eq:x_k_opt_append}) to obtain
\begin{align*}
& \max_{\lambda}\frac{1}{2}x_{S}^{\top}J_{S,k}x_{S}-x_{S}^{\top}\eta_{S,k}-\frac{1}{2}\lambda^{\top}\left(C_{S,k}+S_{k}^{-1}\right)\lambda\\
& +v_{k}^{\top}S_{k}^{-1}v_{k}-\lambda^{\top}\left(S_{k}^{-1}v_{k}-A_{S,k}x_{S}-b_{S,k}\right).
\end{align*}
Making the gradient of this function w.r.t. $\lambda$ equal to zero,
we obtain that the maximum is obtained for
\begin{align}
\lambda & =\left(C_{S,k}+S_{k}^{-1}\right)^{-1}\left(-S_{k}^{-1}v_{k}+A_{S,k}x_{S}+b_{S,k}\right).\label{eq:lamda_trajectory_k}
\end{align}
Substituting (\ref{eq:lamda_trajectory_k}) into (\ref{eq:x_k_trajectory_prev}),
we obtain \eqref{eq:lqt_method_2}, which finishes the proof of Lemma \ref{lem:LQT_trajectory_method2}.

\section*{Acknowledgment}

The authors would like to thank Adrien Corenflos for help in TensorFlow programming.

\bibliographystyle{IEEEtran}
\bibliography{IEEEabrv,dynprog}

\begin{thebibliography}{10}
\providecommand{\url}[1]{#1}
\csname url@samestyle\endcsname
\providecommand{\newblock}{\relax}
\providecommand{\bibinfo}[2]{#2}
\providecommand{\BIBentrySTDinterwordspacing}{\spaceskip=0pt\relax}
\providecommand{\BIBentryALTinterwordstretchfactor}{4}
\providecommand{\BIBentryALTinterwordspacing}{\spaceskip=\fontdimen2\font plus
\BIBentryALTinterwordstretchfactor\fontdimen3\font minus
  \fontdimen4\font\relax}
\providecommand{\BIBforeignlanguage}[2]{{%
\expandafter\ifx\csname l@#1\endcsname\relax
\typeout{** WARNING: IEEEtran.bst: No hyphenation pattern has been}%
\typeout{** loaded for the language `#1'. Using the pattern for}%
\typeout{** the default language instead.}%
\else
\language=\csname l@#1\endcsname
\fi
#2}}
\providecommand{\BIBdecl}{\relax}
\BIBdecl

\bibitem{Stengel_book94}
R.~F. Stengel, \emph{Optimal Control and Estimation}.\hskip 1em plus 0.5em
  minus 0.4em\relax Dover publications, 1994.

\bibitem{Lewis+Syrmos:1995}
F.~L. Lewis and V.~L. Syrmos, \emph{Optimal Control}, 2nd~ed.\hskip 1em plus
  0.5em minus 0.4em\relax John Wiley \& Sons, 1995.

\bibitem{Bertsekas:2005}
D.~P. Bertsekas, \emph{Dynamic Programming and Optimal Control}, 3rd~ed.\hskip
  1em plus 0.5em minus 0.4em\relax Athena Scientific, 2005.

\bibitem{Biggs:2009}
J.~{Biggs} and W.~{Holderbaum}, ``Optimal kinematic control of an autonomous
  underwater vehicle,'' \emph{IEEE Transactions on Automatic Control}, vol.~54,
  no.~7, pp. 1623--1626, 2009.

\bibitem{Komaee:2014}
A.~{Komaee} and A.~{Bensoussan}, ``Optimal control of hidden {M}arkov models
  with binary observations,'' \emph{IEEE Transactions on Automatic Control},
  vol.~59, no.~1, pp. 64--77, 2014.

\bibitem{Zhao:2021}
G.~{Zhao} and M.~{Zhu}, ``Pareto optimal multi-robot motion planning,''
  \emph{IEEE Transactions on Automatic Control}, 2021, in press.

\bibitem{Sutton+Barto:2018}
R.~S. Sutton and A.~G. Barto, \emph{Reinforcement Learning: An Introduction},
  2nd~ed.\hskip 1em plus 0.5em minus 0.4em\relax MIT Press, 2018.

\bibitem{Bellman:1957}
R.~Bellman, \emph{Dynamic Programming}.\hskip 1em plus 0.5em minus 0.4em\relax
  Princeton University Press, 1957.

\bibitem{Bellman+Dreyfus:1962}
R.~E. Bellman and S.~E. Dreyfus, ``Applied dynamic programming,'' The RAND
  Corporation, Tech. Rep. R-352-PR, 1962.

\bibitem{Pakniyat:2017}
A.~{Pakniyat} and P.~E. {Caines}, ``On the relation between the minimum
  principle and dynamic programming for classical and hybrid control systems,''
  \emph{IEEE Transactions on Automatic Control}, vol.~62, no.~9, pp.
  4347--4362, 2017.

\bibitem{Gengler:1996}
M.~Gengler, ``An introduction to parallel dynamic programming,'' in
  \emph{Solving Combinatorial Optimization Problems in Parallel}.\hskip 1em
  plus 0.5em minus 0.4em\relax Springer, 1996, pp. 87--114.

\bibitem{Dormido-Canto05}
S.~{Dormido Canto}, A.~P. {de Madrid}, and S.~{Dormido Bencomo}, ``Parallel
  dynamic programming on clusters of workstations,'' \emph{IEEE Transactions on
  Parallel and Distributed Systems}, vol.~16, no.~9, pp. 785--798, 2005.

\bibitem{Frison13}
G.~Frison and J.~B. J{\o}rgensen, ``Parallel implementation of {R}iccati
  recursion for solving linear-quadratic control problems,'' in
  \emph{Proceedings of the 18th Nordic Process Control Workshop}, 2013.

\bibitem{Axehill15}
D.~Axehill, ``Controlling the level of sparsity in {MPC},'' \emph{Systems and
  Control Letters}, vol.~76, pp. 1--7, Feb. 2015.

\bibitem{Frison16}
G.~Frison, D.~Kouzoupis, J.~B. Jørgensen, and M.~Diehl, ``An efficient
  implementation of partial condensing for nonlinear model predictive
  control,'' in \emph{IEEE 55th Conference on Decision and Control}, 2016, pp.
  4457--4462.

\bibitem{Casti73}
J.~Casti, M.~Richardson, and R.~Larson, ``Dynamic programming and parallel
  computers,'' \emph{Journal of Optimization Theory and Applications}, vol.~12,
  no.~4, pp. 423--438, 1973.

\bibitem{Wright:1991}
S.~J. Wright, ``Partitioned dynamic programming for optimal control,''
  \emph{SIAM Journal on Optimization}, vol.~1, no.~4, pp. 620--642, 1991.

\bibitem{Frash15}
J.~V. Frasch, S.~Sager, and M.~A. Diehl, ``A parallel quadratic programming
  method for dynamic optimization problems,'' \emph{Mathematical Programming
  Computation}, vol.~7, pp. 289--329, 2015.

\bibitem{shin2019parallel}
S.~Shin, T.~Faulwasser, M.~Zanon, and V.~M. Zavala, ``A parallel decomposition
  scheme for solving long-horizon optimal control problems,'' in \emph{2019
  IEEE 58th Conference on Decision and Control (CDC)}.\hskip 1em plus 0.5em
  minus 0.4em\relax IEEE, 2019, pp. 5264--5271.

\bibitem{Jiang:2021}
Y.~Jiang, J.~Oravec, B.~Houska, and M.~Kvasnica, ``Parallel {MPC} for linear
  systems with input constraints,'' \emph{IEEE Transactions on Automatic
  Control}, vol.~66, no.~7, pp. 3401--3408, 2021.

\bibitem{Calvet:1995}
J.~L. Calvet and G.~Viargues, ``Invariant imbedding and parallelism in dynamic
  programming for feedback control,'' \emph{Journal of Optimization Theory and
  Applications}, vol.~87, no.~1, pp. 121--140, 1995.

\bibitem{Blelloch:1989}
G.~E. Blelloch, ``Scans as primitive parallel operations,'' \emph{IEEE
  Transactions on Computers}, vol.~38, no.~11, pp. 1526--1538, 1989.

\bibitem{Blelloch:1990}
------, ``Prefix sums and their applications,'' School of Computer Science,
  Carnegie Mellon University, Tech. Rep. CMU-CS-90-190, 1990.

\bibitem{Boyd_book04}
S.~Boyd and L.~Vandenberghe, \emph{Convex Optimization}.\hskip 1em plus 0.5em
  minus 0.4em\relax Cambridge University Press, 2004.

\bibitem{li2004iterative}
W.~Li and E.~Todorov, ``Iterative linear quadratic regulator design for
  nonlinear biological movement systems,'' in \emph{ICINCO (1)}, 2004, pp.
  222--229.

\bibitem{Abadi_et_al:2015}
\BIBentryALTinterwordspacing
M.~Abadi \emph{et~al.}, ``{TensorFlow}: Large-scale machine learning on
  heterogeneous systems,'' 2015, software available from tensorflow.org.
  [Online]. Available: \url{https://www.tensorflow.org/}
\BIBentrySTDinterwordspacing

\bibitem{Sarkka:2021}
S.~S\"arkk\"a and A.~F. Garc\'ia-Fern\'andez, ``Temporal parallelization of
  {B}ayesian smoothers,'' \emph{IEEE Transactions on Automatic Control},
  vol.~66, pp. 299--306, 2021.

\bibitem{yaghoobi2021parallel}
F.~Yaghoobi, A.~Corenflos, S.~Hassan, and S.~S{\"a}rkk{\"a}, ``Parallel
  iterated extended and sigma-point {K}alman smoothers,'' in \emph{To appear in
  Proceedings of IEEE International Conference on Acoustics, Speech and Signal
  Processing (ICASSP)}, 2021.

\bibitem{Hassan:2021}
S.~Hassan, S.~S\"arkk\"a, and A.~F. Garc\'ia-Fern\'andez, ``Temporal
  parallelization of inference in hidden {M}arkov models,'' \emph{IEEE
  Transactions on Signal Processing}, vol.~69, pp. 4875--4887, 2021.

\bibitem{Dower:2015}
P.~M. Dower, W.~M. McEneaney, and H.~Zhang, ``Max-plus fundamental solution
  semigroups for optimal control problems,'' in \emph{2015 Proceedings of the
  Conference on Control and its Applications}.\hskip 1em plus 0.5em minus
  0.4em\relax SIAM, 2015, pp. 368--375.

\bibitem{Zhang:2015}
H.~Zhang and P.~M. Dower, ``Max-plus fundamental solution semigroups for a
  class of difference {R}iccati equations,'' \emph{Automatica}, vol.~52, pp.
  103--110, 2015.

\bibitem{Zhang:2015b}
------, ``A max-plus based fundamental solution for a class of discrete time
  linear regulator problems,'' \emph{Linear Algebra and its Applications}, vol.
  471, pp. 693--729, 2015.

\bibitem{Xu:2019}
J.~{Xu}, T.~{van den Boom}, and B.~{De Schutter}, ``Model predictive control
  for stochastic max-plus linear systems with chance constraints,'' \emph{IEEE
  Transactions on Automatic Control}, vol.~64, no.~1, pp. 337--342, 2019.

\bibitem{LaValle_book06}
S.~M. La{V}alle, \emph{Planning Algorithms}.\hskip 1em plus 0.5em minus
  0.4em\relax Cambridge University Press, 2006.

\bibitem{Maybeck:1982b}
P.~S. Maybeck, \emph{Stochastic Models, Estimation and Control}.\hskip 1em plus
  0.5em minus 0.4em\relax New York, NY: Academic Press, 1982, vol.~3.

\bibitem{Rauber:2013}
T.~Rauber and G.~R{\"u}nger, \emph{Parallel programming: For multicore and
  cluster systems}, 2nd~ed.\hskip 1em plus 0.5em minus 0.4em\relax Springer,
  2013.

\bibitem{Barlas:2015}
G.~Barlas, \emph{Multicore and {GPU} Programming: An Integrated
  Approach}.\hskip 1em plus 0.5em minus 0.4em\relax Morgan Kaufmann Publishers
  Inc., 2015.

\bibitem{Apostol_book67}
T.~M. Apostol, \emph{Calculus. Volume I.}\hskip 1em plus 0.5em minus
  0.4em\relax John Wiley \& Sons, 1967.

\bibitem{Koller_book09}
D.~Koller and N.~Friedman, \emph{Probabilistic Graphical Models: Principles and
  Techniques}.\hskip 1em plus 0.5em minus 0.4em\relax The MIT Press, 2009.

\bibitem{Arora_book07}
S.~Arora and B.~Barak, \emph{Computational Complexity: A Modern
  Approach}.\hskip 1em plus 0.5em minus 0.4em\relax Cambridge University Press,
  2007.

\bibitem{Ortega:1988}
J.~M. Ortega, \emph{Introduction to parallel and vector solution of linear
  systems}.\hskip 1em plus 0.5em minus 0.4em\relax Springer Science \& Business
  Media, 1988.

\bibitem{Englert:2019}
T.~Englert, A.~V\"olz, F.~Mesmer, S.~Rhein, and K.~Graichen, ``A software
  framework for embedded nonlinear model predictive control using a
  gradient-based augmented {L}agrangian approach ({GRAMPC}),''
  \emph{Optimization and Engineering}, vol.~20, pp. 769--809, 2019.

\bibitem{kapernick2016gradient}
B.~K{\"a}pernick, ``Gradient-based nonlinear model predictive control with
  constraint transformation for fast dynamical systems,'' Ph.D. dissertation,
  Universit{\"a}t Ulm, 2016.

\bibitem{Nisio:2015}
M.~Nisio, \emph{Stochastic control theory: Dynamic programming
  principle}.\hskip 1em plus 0.5em minus 0.4em\relax Springer, 2015.

\end{thebibliography}

\begin{IEEEbiography}[{\includegraphics[width=1in,height=1.25in,clip,keepaspectratio]{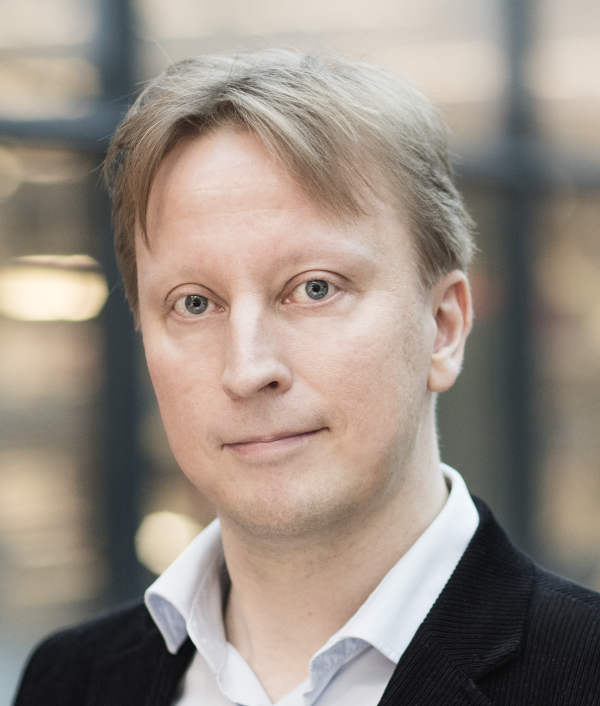}}]{Simo S\"arkk\"a} received his Master of Science (Tech.) degree (with distinction) in engineering physics and mathematics, and Doctor of Science (Tech.) degree (with distinction) in electrical and communications engineering from Helsinki University of Technology, Espoo, Finland, in 2000 and 2006, respectively. Currently, Dr. S\"arkk\"a is an Associate Professor with Aalto University and an Adjunct Professor with Tampere University of Technology and Lappeenranta University of Technology. His research interests are in multi-sensor data processing systems with applications in location sensing, health and medical technology, machine learning, inverse problems, and brain imaging. He has authored or coauthored over 150 peer-reviewed scientific articles and his books "Bayesian Filtering and Smoothing" and "Applied Stochastic Differential Equations" along with the Chinese translation of the former were recently published via the Cambridge University Press. He is a Senior Member of IEEE and serving as an Senior Area Editor of IEEE Signal Processing Letters.  \end{IEEEbiography}

\begin{IEEEbiography}[{\includegraphics[width=1in,height=1.25in,clip,keepaspectratio]{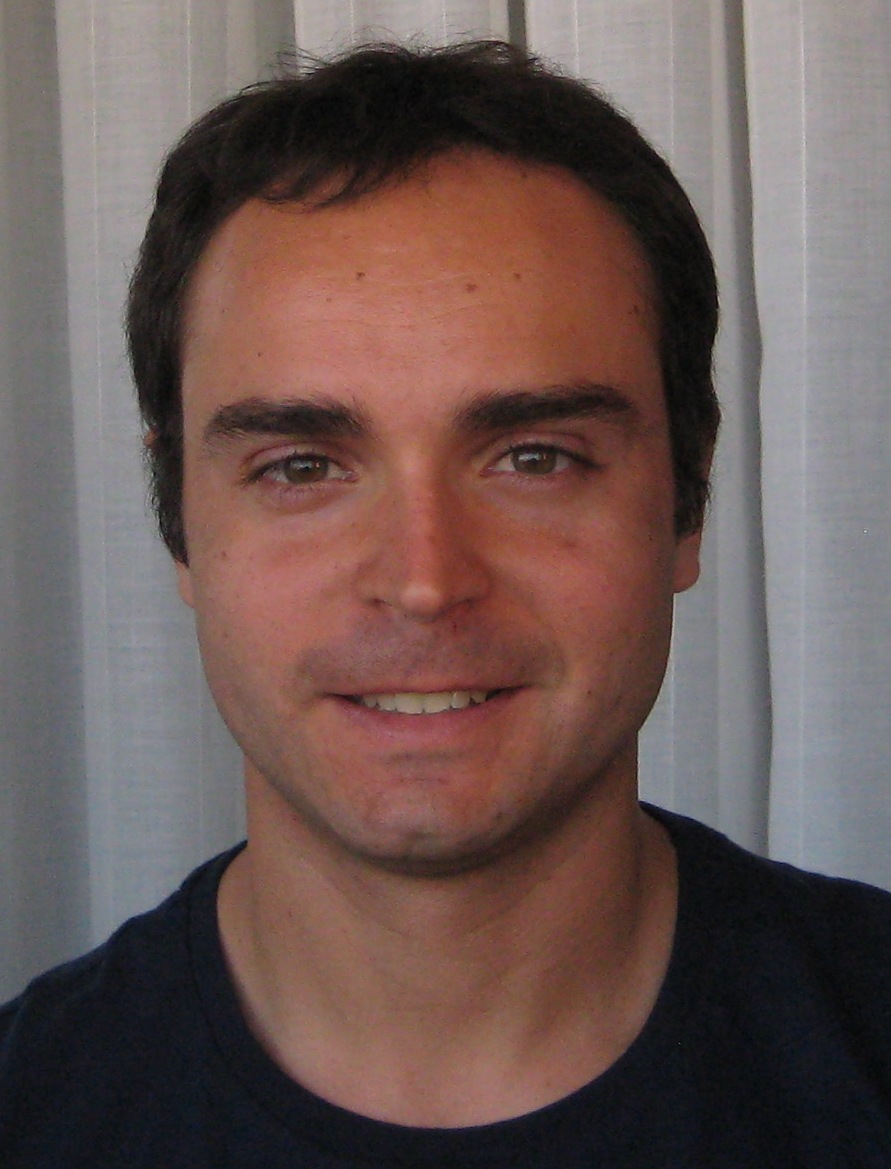}}]{\'Angel F. Garc\'ia-Fern\'andez} received the telecommunication engineering degree (with honours) and the Ph.D. degree from Universidad Polit\'ecnica de Madrid, Madrid, Spain, in 2007 and 2011, respectively. \\ 
	He is currently a Lecturer in the Department of Electrical Engineering and Electronics at the University of Liverpool, Liverpool, UK. He previously held postdoctoral positions at Universidad Polit\'ecnica de Madrid, Chalmers University of Technology, Gothenburg, Sweden, Curtin University, Perth, Australia, and  Aalto University, Espoo, Finland. His main research activities and interests are in the area of Bayesian  estimation, with emphasis on dynamic systems and multiple target tracking. 
	He was the recipient of paper awards at the International Conference on Information Fusion in 2017, 2019 and 2021. \end{IEEEbiography}

\vfill

\end{document}